\documentclass[DIV=13,abstract=true,paper=a4,fontsize=11pt,parskip=half]{scrartcl}

\RequirePackage{booktabs}
\RequirePackage{dsfont}
\RequirePackage{microtype}
\usepackage{lmodern}
\usepackage{inputenc}
\usepackage[T1]{fontenc}
\usepackage[USenglish]{babel}
\usepackage{latexsym}
\usepackage{amsmath,amsfonts,amssymb,amsthm,amsopn}
\usepackage[centercolon]{mathtools}
\usepackage{fixmath} 
\usepackage{braket}
\usepackage{graphicx}
\usepackage{epstopdf}
\usepackage[svgnames]{xcolor}
\usepackage{tikz}
\usepackage{pgfplots,pgfplotstable}
\usepackage[normalem]{ulem}

\pgfplotsset{compat=1.7}
\usepgfplotslibrary{statistics}
\usetikzlibrary{pgfplots.statistics}
\usepackage{colortbl}
\usepackage{multirow}
\usepackage{array}
\usepackage{makecell}
\usepackage{enumerate}
\usepackage[pdftex,pagebackref=true,bookmarks=true,pdfpagelabels=true,plainpages=false,
hypertexnames=false,bookmarksnumbered=true,bookmarksopen=true,pdfa = true,
pdfencoding=auto,pdfstartview=Fit,colorlinks,
linkcolor=DarkGreen,citecolor=DarkBlue,urlcolor=DarkRed]{hyperref}
\usepackage[algo2e,ruled,linesnumbered]{algorithm2e}
\SetKw{KwTerminate}{terminate}

\SetCommentSty{mycommfont}
\usepackage[nameinlink,capitalise]{cleveref}
\crefformat{equation}{(#2#1#3)}
\crefname{assumption}{Assumption}{Assumptions}
\Crefname{assumption}{Assumption}{Assumptions}
\crefname{figure}{Figure}{Figures}
\theoremstyle{plain}
\newtheorem{theorem}{Theorem}[section]
\newtheorem{lemma}[theorem]{Lemma}
\newtheorem{corollary}[theorem]{Corollary}
\theoremstyle{definition}
\newtheorem{definition}[theorem]{Definition}
\newtheorem{assumption}[theorem]{Assumption}

\theoremstyle{remark}
\newtheorem{remark}[theorem]{Remark}

\providecommand{\MSC}[1]{\textbf{Mathematics Subject Classification} #1}

\newdimen\fwd

\DeclareMathOperator*{\upd}{Update}

\newcommand{\norm}[1]{\ensuremath{\lVert #1\rVert}}

\newcommand{\xopt}{\ensuremath{x^\ast}}
\newcommand{\Lo}{{\L}ojasiewicz}
\newcommand{\KL}{Kurdyka--\Lo}
\newcommand{\minres}{\textsc{MINRES}}

\newcommand{\Ballop}[1]{\ensuremath{\mathbb{B}_{#1}}}
\newcommand{\Lin}{\ensuremath{\mathcal{L}}}
\newcommand{\tl}{\ensuremath{\epsilon}}
\newcommand{\brk}{\ensuremath{break}}
\newcommand{\outp}{\ensuremath{output}}
\newcommand{\CD}{{\cal D}}
\newcommand{\CS}{{\cal S}}
\newcommand{\CJ}{{\cal J}}
\newcommand{\CI}{{\cal I}}
\newcommand{\CJopt}{{\cal J^\ast}}
\newcommand{\CX}{{\cal X}}
\newcommand{\CN}{{\cal N}}
\newcommand{\R}{{\mathbb{R}}}
\newcommand{\N}{\mathbb{N}} 
 
\newcommand{\AlgTu}{Algorithm~TULIP} 
\newcommand{\Tu}{TULIP} 
\newcommand{\AlgRo}{Algorithm~ROSE} 
\newcommand{\Ro}{ROSE} 
\newcommand{\lam}{\ensuremath{\lambda}} 
\newcommand{\Lam}{\ensuremath{\Lambda}}

\newcommand{\Bp}{\ensuremath{s}}
\newcommand{\Bz}{\ensuremath{z}}
\newcommand{\GM}{\ensuremath{g}}

\newcommand{\FAIRTEXT}{{\textsc{FAIR}}}
\newcommand{\LBFGS}{\mbox{L-BFGS}}
\newcommand{\BB}{\mbox{Barzilai--Borwein}}
\newcommand{\eps}{\varepsilon}

\title{A structured L-BFGS method with diagonal scaling and its application to image registration}
\author{Florian Mannel\thanks{Institute of Mathematics and Image Computing, University of Lübeck, Maria-Goeppert-Straße 3, 23562 
		Lübeck, Germany (\href{mailto:florian.mannel@uni-luebeck.de}{florian.mannel@uni-luebeck.de}/\href{mailto:hariom85@gmail.com}{hariom85@gmail.com})}
		\,\href{https://orcid.org/0000-0001-9042-0428}{\includegraphics[height=.35cm]{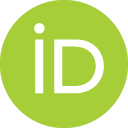}}
	\and Hari Om Aggrawal\footnotemark[1]\, \href{https://orcid.org/0000-0002-8892-5993}{\includegraphics[height=.35cm]{ORCIDiD128x128.png}}}

\date{Preprint, \today}

\ifpdf
\hypersetup{
	pdftitle={A structured L-BFGS method with diagonal scaling and its application to medical image registration},
	pdfauthor={F. Mannel}
	pdfsubject={Accelerating L-BFGS by using readily accessible Hessian information and diagonal scaling}
}
\fi

\begin{document}

\maketitle

\begin{abstract}
	We devise an \LBFGS~method for optimization problems in which the objective is the sum of two functions, 
	where the Hessian of the first function is computationally unavailable while the Hessian of the second function
	has a computationally available approximation that allows for cheap matrix-vector products. 
	This is a prototypical setting for many inverse problems. 
	The proposed \LBFGS~method exploits the structure of the objective to construct a more accurate Hessian approximation than in standard \LBFGS. 
	In contrast to existing works on structured \LBFGS, we choose the first part of the seed matrix, which approximates the Hessian of the first function,
	as a diagonal matrix rather than a multiple of the identity. 
	We derive two suitable formulas for the coefficients of the diagonal matrix and show that this boosts performance on real-life image registration problems,
	which are highly non-convex inverse problems.
	The new method converges globally and linearly on non-convex problems under mild assumptions in a general Hilbert space setting, making it applicable to a broad class of inverse problems. An implementation of the method is freely available. 
\end{abstract}


\MSC{65J22 $\cdotp$ 65K05 $\cdotp$ 65K10 $\cdotp$ 90C06 $\cdotp$ 90C26 $\cdotp$ 90C30 $\cdotp$ 90C48 $\cdotp$ 90C53 $\cdotp$ 90C90}

%65J22 Numerical solution to inverse problems in abstract spaces
%65K05 Numerical mathematical programming methods 
%65K10 Numerical optimization and variational techniques 
%90C06 Large-scale problems in mathematical programming
%90C26 non-convex programming, global optimization 
%90C30 Nonlinear programming
%90C48 Programming in abstract spaces 
%90C53 Methods of quasi-Newton type
%90C90 Applications of mathematical programming

%%%%%%%%%%%%%%%%%%%%%%%%%%%%%%%%%%%%%%%%%%%%%%%%%%%%%%%%%%%%%%%%%%%%%%%%%%%%%%%%%%%%%%%%%%%%%%%%%%%%%%%%%%%%%%%%%%%%%%%%%%
%%%%%%%%%%%%%%%%%%%%%%%%%%%%%%%%%%%%%%%%%%%%%%%%%%%%%%%%%%%%%%%%%%%%%%%%%%%%%%%%%%%%%%%%%%%%%%%%%%%%%%%%%%%%%%%%%%%%%%%%%%

\section{Introduction}
	
	\subsection{Topic}\label{sec_topic}
	
	In this paper we study a new \LBFGS-type method for unconstrained optimization problems 
	\begin{equation*}
		\min_{x\in\CX}\,\CJ(x)
	\end{equation*}
	with a cost function of the form
	\begin{equation}\label{eq_SO}
		\CJ:\CX\to\R,\qquad \CJ(x)=\CD(x)+\CS(x).
	\end{equation}
	Among others, this is a prototypical setting for inverse problems, where $\CD:\CX\rightarrow\R$ represents a data-fitting term, $\CS:\CX\rightarrow\R$ a regularizer, and $\CX$ a Hilbert space. 
	In this setting it is often the case that, after discretization, the data-fitting term has an ill-conditioned and dense Hessian for which even matrix-vector multiplications are computationally expensive, whereas the Hessian of the regularizer is positive definite, well-conditioned and sparse with computationally cheap matrix-vector products. The \LBFGS~method \cite{N80,LN89,BNS94} is one of the most widely used algorithms for large-scale inverse problems, but it
	does not take advantage of the splitting in~\cref{eq_SO}. In our recent work \cite{AMM24} we introduced an \LBFGS-type method called \Tu~that exploits the 
	different structural properties of the two terms in \cref{eq_SO}. We proved the method's global and linear convergence including for the case that the cost function $\CJ$ is non-convex with singular Hessian, and we demonstrated in numerical experiments that it outperforms standard \LBFGS~as well as other structured \LBFGS-type methods on real-world image registration problems. 
	
	In this paper we present \Ro, an algorithm that further improves the numerical performance of \Tu~for image registration problems while
	offering identical convergence guarantees. 
	The improvement, which we demonstrate in the numerical experiments in \cref{sec:experiments}, 
	can be attributed primarily to the fact that in \Ro~the Hessian of $\CD$ is approximated in the seed matrix by a diagonal matrix rather than a multiple of the identity. Let us explain this in more detail. 
	
	In the classical \LBFGS~method with memory length $\ell\in\N_0$, 
	the \LBFGS~operator $B_{k}\approx\nabla^2\CJ(x_{k})$ is obtained as $B_{k}:=B_{k}^{(\ell)}$ from the recursion
	\begin{equation*}
		B_{k}^{(j+1)}:=B_{k}^{(j)} + \upd\bigl(s_{m+j},y_{m+j},B_{k}^{(j)}\bigr), \qquad j=0,\ldots,\ell-1.
	\end{equation*}
	Here, $m:=\max\{0,k-\ell\}$, the stored update vectors are $\{(s_{j},y_{j})\}_{j=m}^{k-1}$, where
	$s_j:=x_{j+1}-x_j$ and $y_j:=\nabla\CJ(x_{j+1})-\nabla\CJ(x_j)$, and they satisfy $y_j^T s_j>0$ for all $j$.
	For $(s,y)\in\CX\times\CX$ with $y^T s>0$ and positive definite $B$ the update is given by 
	\begin{equation*}
		\upd\bigl(s,y,B\bigr):=\frac{y y^T}{y^T s} - \frac{B s s^T B^T}{s^T B s}.
	\end{equation*}
	It is a great advantage of the classical \LBFGS~method that if the seed matrix $B_{k}^{(0)}$ is a multiple of the identity $\tau_k I$ for some $\tau_k>0$, then
	the search direction $d_k=-B_k^{-1}\nabla\CJ(x_k)$ can be computed matrix free and without having to solve a linear system. In practice, this is efficiently realized through the \emph{two-loop recursion} (e.g. \cite[Fig.~1]{BB19}, \cite[Algorithm~7.4]{NW06}), enabling the use of \LBFGS~for large-scale problems. 
	On the other hand, this choice of the seed matrix does not take into account the structure \cref{eq_SO} and, in turn, does not benefit from the convenient properties of $\nabla^2 \CS(x_k)$. To change this, in \Tu~the seed matrix in iteration $k$ is taken to be
	\begin{equation*}
		B_k^{(0)} = \tau_k I + S_k,
	\end{equation*}
	where $S_k$ approximates $\nabla^2\CS(x_k)$ and is selected in such a way that $B_k^{(0)}$ is positive definite and linear systems involving $B_k^{(0)}$ can be solved cheaply at least approximately. 
	While the choice $B_k^{(0)} = \tau_k I + S_k$ is expected to make $B_k$ a better approximation of $\nabla^2\CJ(x_k)$ and improve the rate of convergence, 
	the computation of $d_k=-B_k^{-1}\nabla\CJ(x_k)$ now requires the solution of a linear system involving $B_k^{(0)}$. 
	In this paper, we consider seed matrices of the more general form 
	\begin{equation*}
		B_k^{(0)} = D_k + S_k,
	\end{equation*}
	where $D_k$ and $S_k$ are chosen in such a way that $B_k^{(0)}$ is positive definite and linear systems involving $B_k^{(0)}$ can be solved cheaply at least approximately. It is clear that this generalizes \Tu.
	The main focus in this paper is on the choice of $D_k$ as a diagonal matrix. In particular, we propose two formulas for the entries of the diagonal matrix in \cref{sec_SLBFGS} and we compare them numerically in \cref{sec:experiments}. We also use the convergence theory developed in \cite{AMM24} for \Tu~to obtain convergence results for \Ro~in \cref{sec_convana}. 
	
	In view of the structure \cref{eq_SO}, the matrices $\tau_k I$ and $S_k$, respectively, $D_k$ and $S_k$, may be regarded as approximations of the Hessians $\nabla^2\CD(x_k)$ and $\nabla^2\CS(x_k)$, respectively.
	It is therefore expected that the approximation quality of $B_k\approx\nabla^2\CJ(x_k)$ increases from \LBFGS~to \Tu~to \Ro.
	Consequently, \LBFGS~should usually require more iterations than \Tu, which should require more iterations than \Ro. 
	Since the increase in approximation quality in the structured methods \Tu~and \Ro~comes at the cost of (inexactly) solving one linear system per iteration, the question arises whether the structured approach actually lowers the run-time in comparison to standard \LBFGS. 
	In \cite{AMM24} the answer was affirmative for \Tu~when we considered a test set of 22~real-world problems from medical image registration. 
	In the present paper we find that \Ro~is significantly faster than \Tu~on the same set of test problems. 
	These problems are large-scale and highly non-convex inverse problems that involve various data-fitting terms and regularizers, indicating that 
	\Ro~is a promising method also for other inverse problems. 			
			
	\subsection{Related work}\label{lit_relwork}
	
	The idea to exploit the problem structure for constructing a better approximation of $\nabla^2\CJ(x_k)$ and use it as a seed matrix in \LBFGS~appears in the numerical studies \cite{JBES04,Heldmann2006,FAIR09,KR13,ACG18,YGJ18,AM21,LBLPT21,BDLP22,AMM24} that address
	a wide range of real-life problems. 
	On the considered large-scale problems, the authors report significant speed ups over all methods that are used for comparison,
	including standard \LBFGS, Gauss--Newton and truncated Newton.
	Among those contributions, only \cite{YGJ18,LBLPT21} employ a diagonal seed matrix other than a scaled identity, but they do not consider the structure \cref{eq_SO}. Thus, to the best of our knowledge, \AlgRo~is the first structured \LBFGS~method with proper diagonal scaling. Another important difference to the present paper is that except for \cite{AMM24}, convergence rates are only studied numerically.

	The convergence results of this work also hold for memory size zero. 
	In this case, no updates are applied, hence $B_k=B_k^{(0)}=D_k + S_k$. 
	For the choice $S_k=0$ this may be regarded as a method with \emph{diagonal \BB~step size} \cite{PDBS20}. 
	Again, however, it seems that this work is the first that integrates diagonal \BB~step sizes into a \emph{structured} method.
	We emphasize that the two choices proposed for $D_k$ in \cref{sec_SLBFGS} are inspired by \cite{PDBS20},
	but are still somewhat different. 
	
	For $\CS\equiv 0$ and $S_k=0$ for all $k$, \Ro~is an unstructured \LBFGS~method that uses a diagonal seed matrix. This class of methods has been studied in \cite{LN89,Gilbert1989,VAF00,SK09,ML13,LC13,B15,Dener2019,LWH22,BKAG22}, but convergence is usually shown for strongly convex objectives or not at all, except in \cite{LWH22}, where global convergence is proved for the non-convex case in the sense that $\liminf_{k\to\infty}\norm{\nabla\CJ(x_k)}=0$. In contrast, we have $\lim_{k\to\infty}\norm{\nabla\CJ(x_k)}=0$ in that case and a linear rate of convergence, cf.  \cref{thm_globconv,thm_rateofconvPL,thm_rateofconvstrongminimizer}. That is, in addition to its contribution in the structured setting, which is the main focus of this work, \Ro~also provides convergence guarantees that are stronger than existing ones for unstructured objectives. 
	Non-diagonal seed matrices have also been considered, for instance in \cite{AB99}, where the seed matrix itself contains low-rank updates.
	 
	In \emph{diagonal quasi-Newton methods} the Hessian is approximated by a diagonal matrix. 
	In contrast to an \LBFGS~based approach, however, low-rank updates are not applied to the diagonal matrix.
	References include \cite{ZNW99,LFH12,ELF16,An18,An21,LEK21,ABK23},
	but we are not aware of works that embed this approach in a structured method. 
	
	Despite the large amount of works that involve diagonal matrices, it seems that one of the two choices that we propose for $D_k$ in \cref{sec_SLBFGS} has not been considered before, while the other one may be regarded as a generalization of the diagonal matrix that appears in \cite{An19} within an unstructured diagonal quasi-Newton method. As mentioned above, however, there are no low-rank updates applied to the diagonal matrix in \cite{An19}, which differs from our approach. The numerical experiments in \cite{An19} show that the method requires less CPU time than BFGS, but more than \LBFGS. 
	On the theoretical side, \cite{An19} proves global convergence for strictly convex quadratics. 
		
	Structured variants of \emph{full memory} quasi-Newton methods have been studied 
	in various settings, cf. e.g. \cite{DW81,DW85,EM91,YY96,DMT89,AR10,La00,HK92,MR21,MR22},
	most often in the context of least squares problems \cite{DMT89,H94,H05,ZC10,MM19}. 
	However, they do not allow for a seed matrix, so by design they are somewhat different from \AlgRo. 
		
	\subsection{Main contributions}
		
	The main contributions of this paper are that 
	\begin{itemize}
			\item we present \Ro, the first structured \LBFGS~method with diagonal scaling.
			Additionally, the specific diagonal scaling that we propose is new even for unstructured \LBFGS;
			\item we obtain global and linear convergence of \Ro~for non-convex problems 
			without assuming invertibility of the Hessian, cf. \cref{thm_globconv,thm_rateofconvPL}.
			Such strong results are not available for other structured \LBFGS~methods except for our own
			method \Tu~\cite{AMM24}; 
			\item we show that \Ro~outperforms \Tu~on real-world image registration problems. 
			This is significant since \Tu~outperforms standard \LBFGS~as well as competing structured \LBFGS~methods on the same set of problems \cite{AMM24};
			\item we work in Hilbert space, which is rarely done for \LBFGS~and also for diagonal scaling.
			This is valuable for instance because infinite dimensional Hilbert spaces are a natural setting for many inverse problems.
	\end{itemize}
	
	\subsection{Code availability}
	
	An implementation of our structured \LBFGS~method that includes an example from the numerical section of this paper
	is freely available at \href{https://github.com/hariagr/SLBFGS}{https://github.com/hariagr/SLBFGS}.
		
	\subsection{Organization and notation}
	
	The paper is organized as follows. 
	In \cref{sec_SLBFGS} we introduce \Ro. 
	\Cref{sec_convana} collects the convergence results 
	and \cref{sec:experiments} contains the numerical experiments. 
	Conclusions are drawn in \cref{sec:conclusion}.
	
	We use $\N=\{1,2,3,\ldots\}$ and $\N_0=\N\cup\{0\}$.
	The scalar product of $x,y\in\CX$ is indicated by $x^T y$, and for $x\in\CX$ the linear functional
	$y\mapsto x^T y$ is denoted by $x^T$. 
	The induced norm is $\norm{x}$. 
	We write $M\in\Lin(\CX)$ if $M:\CX\rightarrow\CX$ is 
	a bounded linear mapping. The notation $M\in\Lin_{\geq 0}(\CX)$ means that 
	$M\in\Lin(\CX)$ is symmetric and positive semi-definite. 
	For $M\in\Lin(\CX)$ we define
	\begin{equation*}
		\lam(M):=\inf_{\norm{v}=1}\norm{Mv}\qquad\text{ and }\qquad\Lam(M):=\sup_{\norm{v}=1}\norm{Mv}.
	\end{equation*}
	Apparently, we have $0\leq\lam(M)\leq\Lam(M)$. If $M\in\Lin_{\geq 0}(\CX)$, then $\lam(M)=\inf_{\norm{v}=1}v^T Mv$ and $\Lam(M)=\norm{M}=\sup_{\norm{v}=1}v^T Mv$. 
	Furthermore, if $M\in\Lin_{\geq 0}(\CX)$ is positive definite, then it is invertible, $M^{-1}$ is symmetric and $\lam(M^{-1})=\Lam(M)^{-1}>0$ as well as $\Lam(M^{-1}) = \lam(M)^{-1}>0$. If $\CX$ is finite dimensional, then $\lam(M)$ and $\Lam(M)$ are the smallest and largest eigenvalue of $M\in\Lin_{\geq 0}$.
	
%%%%%%%%%%%%%%%%%%%%%%%%%%%%%%%%%%%%%%%%%%%%%%%%%%%%%%%%%%%%%%%%%%%%%%%%%%%%%%%%%%%%%%%%%%%%%%%%%%%%%%%%%%%%%%%%%%%%%%%%%%
%%%%%%%%%%%%%%%%%%%%%%%%%%%%%%%%%%%%%%%%%%%%%%%%%%%%%%%%%%%%%%%%%%%%%%%%%%%%%%%%%%%%%%%%%%%%%%%%%%%%%%%%%%%%%%%%%%%%%%%%%%

\section{The structured \LBFGS~method \Ro}\label{sec_SLBFGS}
	
	The proposed structured \LBFGS~algorithm with seed matrix $B_k^{(0)}=D_k+S_k$ 
	is summarized as Algorithm~\ref{alg_ROSE} (St\underline{r}uctured inverse \LBFGS~meth\underline{o}d with flexible \underline{se}ed matrix) below. 
	
	\begin{algorithm2e}
		\SetAlgoRefName{ROSE}
		\DontPrintSemicolon
		\caption{Structured inverse \LBFGS~method with flexible seed matrix}
		\label{alg_ROSE}
		\KwIn{$x_0\in \CX$, $\tl\geq 0$, $\ell\in\N_0$, $c_0\geq 0$, $C_0\in [c_0,\infty]$, $c_s,c_1,c_2>0$}
		Choose $S_0\in\Lin_{\geq 0}(\CX)$\\
		Let $D_0 = \tau I$ for some $\tau>0$\\
	\For{$k=0,1,2,\ldots$}
	{
	Let $m:=\max\{0,k-\ell\}$\\
	Let $B_k^{(0)}:=D_k+S_k$\label{line_choiceofseed}\tcp*{choice of seed matrix}
	Compute $d_k:=-B_k^{-1}\nabla\CJ(x_k)$ from $B_k^{(0)}$ and the stored pairs $\{(s_j,y_j)\}_{j=m}^{k-1}$ using the two-loop recursion \cite[Algorithm~7.4]{NW06}\label{line_tlr}\\
	Compute step length $\alpha_k>0$ using a line search\\ 
	Let $s_k:=\alpha_k d_k$, $x_{k+1}:=x_k + s_k$, $y_k:=\nabla\CJ(x_{k+1})-\nabla\CJ(x_k)$\\
	\lIf{$y_k^T s_k > c_s\norm{s_k}^2$}{append $(s_k,y_k)$ to storage\label{line_acceptanceofysforstorage}}\vspace*{-\baselineskip}
	\tcp*[f]{cautious update~1}\\
	\lIf{$k\geq\ell$}{remove $(s_m,y_m)$ from storage}
	\lIf{$\norm{\nabla\CJ(x_{k+1})}\leq\tl$}{\outp~$x_{k+1}$ and \brk\label{line_termination}}
	Choose $S_{k+1}\in\Lin_{\geq 0}(\CX)$\label{line_Skpossemdef}\\ 
	Let $z_k:=y_k-S_{k+1}s_k$\label{line_zk}\\
	Let $\omega_{k+1}^l:=\min\{c_0,c_1\norm{\nabla\CJ(x_{k+1})}^{c_2}\}$ and $\omega_{k+1}^u:=\max\{C_0,(c_1\norm{\nabla\CJ(x_{k+1})}^{c_2})^{-1}\}$\\
	\lIf{$z_k^Ts_k>0$}{let $T_{k+1}:=\bigl[\omega_{k+1}^l,\omega_{k+1}^u\bigr]$ \textbf{else} let $T_{k+1}:=\bigl[\omega_{k+1}^l,P_{k+1}(\tau_{k+1}^\GM)\bigr]$\label{line_choicetauk1}}
	Choose $D_{k+1}\in\Lin_{\geq 0}(\CX)$ such that $\lam(D_{k+1}),\Lam(D_{k+1})\in T_{k+1}$\label{line_choiceofDkp1}\tcp*[f]{cautious update~2}
	}
	\end{algorithm2e}
	
	Here, $P_{k+1}:\R\rightarrow\R$ projects onto $[\omega_{k+1}^l,\omega_{k+1}^u]$, i.e., 
	$P_{k+1}(t):=\min\{\omega_{k+1}^u,\max\{\omega_{k+1}^l,t\}\}$.
	The scalar $\tau_{k+1}^\GM$ is defined in \cref{def_tau_B} below. 
	
	The main differences between Algorithm~\ref{alg_ROSE} and standard \LBFGS~are the choice of the seed matrix $B_k^{(0)}$ 
	and the two cautious updates that affect if $(s_k,y_k)$ enters the storage and that restrict the choice of $D_{k+1}$.  
	We stress that the two-loop recursion in \cref{line_tlr} requires solving a linear system whose system matrix is $B_k^{(0)}$.
		In standard \LBFGS, where $B_k^{(0)}=\tau_k I$, this is a trivial task, but in \ref{alg_ROSE}, where $B_k^{(0)}=D_k+S_k$, 
		this is usually non-trivial.  
		As discussed in \cref{sec_topic}, it is therefore important to choose $D_k$ and $S_k$ in such a way that linear systems with $B_k^{(0)}$ are relatively cheap to solve either directly or iteratively. In the numerical experiments we choose $D_k$ as a diagonal matrix with (sufficiently) positive entries, so for many regularizers $\CS$ the choice $S_k=\nabla^2\CS(x_k)$ will result in linear systems that are cheap to solve.
		In \cref{line_choiceofDkp1} we see that the interval $T_{k+1}$ is used to restrict the spectrum of $D_{k+1}$, which is called \emph{cautious updating}. Cautious updating is critical for the strong convergence properties of \ref{alg_ROSE}; cf. \cref{sec_cautiousupdating} for details. 
			
	In comparison to Algorithm~TULIP from \cite{AMM24}, \ref{alg_ROSE} is more general. Specifically, 
	the interval $T_{k+1}$ is larger in \ref{alg_ROSE} and $D_{k+1}$ in TULIP is restricted to multiples of the identity. 
	In the following subsections we offer two choices for $D_{k+1}$ and provide further commentary on \ref{alg_ROSE}.

	\subsection{Choice of \texorpdfstring{$D_k$}{D_k}}\label{sec_choiceoftaukinLBFGSM}
	
	A key element of \ref{alg_ROSE} is to use 
	\begin{equation*}
		B_{k+1}^{(0)} = D_{k+1} + S_{k+1}
	\end{equation*}
	as seed matrix, where $D_{k+1}$ is a symmetric positive semi-definite operator. 
	In classical \LBFGS~we have $D_{k+1}=\tau_{k+1} I$ for some $\tau_{k+1}>0$ and $S_{k+1}=0$ for all $k$,
	with $\tau_{k+1}=y_k^T s_k/\norm{y_k}^2$ being the most popular choice.
	This choice as well as some others can be derived from the Oren--Luenberger scaling strategy~\cite{Oren1982}, 
	which postulates that $B_{k+1}^{(0)}$ should satisfy the secant equation 
	$y_k=B_{k+1}^{(0)} s_k$ in a least squares sense. In the structured setting of this paper, 
	where $B_{k+1}^{(0)}=D_{k+1} + S_{k+1}$, the secant equation reads
	\begin{equation}\label{eq_Bk0Secant}
		D_{k+1} s_k-z_k=0,
		\qquad\text{where}\qquad 
		z_k:=y_k-S_{k+1} s_k.
	\end{equation}
	If $D_{k+1}$ is invertible, this equation has equivalent forms such as
	$D_{k+1}^{1/2}s_k-D_{k+1}^{-1/2}z_k=0$ or $s_k-D_{k+1}^{-1}z_k=0$.
	This motivates to choose $D_{k+1}$ in such a way that it minimizes the associated least squares problem, e.g. 
	$\norm{D_{k+1}s_k-z_k}$ or $\norm{D_{k+1}^{1/2}s_k-D_{k+1}^{-1/2}z_k}$. 
	On the other hand, it is well known that to maintain positive definiteness of $B_{k+1}$, the seed matrix $B_{k+1}^{(0)}$ has to be positive definite. In fact, to prove convergence of the algorithm, the positive definiteness of $B_{k+1}$ is not strong enough; it is also necessary to appropriately control the condition number of $B_{k+1}$ so that it does not go to infinity too quickly.
	For the problems that we are interested in, it is reasonable to expect that the approximations $S_{k+1}$ of the Hessians $\nabla^2\CS(x_{k+1})$ have well-behaved condition numbers.	
	In this setting, due to \cref{line_acceptanceofysforstorage}, the condition number of $B_{k+1}$ can be controlled 
	by ensuring that $\lam(D_{k+1})$ and $\Lam(D_{k+1})$ belong to an appropriate interval, denoted $T_{k+1}$ in \ref{alg_ROSE}, cf. \cref{line_choicetauk1}. 
	We comment further on $T_{k+1}$ when we discuss cautious updating in \cref{sec_cautiousupdating}.
	
	As it turns out, the convergence analysis of \cite{AMM24} goes through for Algorithm~\ref{alg_ROSE} without further specification of $D_{k+1}$. 
	That is, the convergence results that we obtain hold for any $D_{k+1}\in\Lin_{\geq 0}(\CX)$ satisfying 
	$\lambda(D_{k+1}),\Lambda(D_{k+1})\in T_{k+1}$. 
	However, to make the method efficient in practical computations it is crucial that linear systems with $B_{k+1}^{(0)}$ can be solved, at least approximately, in an efficient way. 
	Thus, although the convergence analysis does not require a specific structure of $D_{k+1}$, we are mainly interested in choices that lead to a seed matrix $B_{k+1}^{(0)}$ with favorable properties (for iterative solvers). Recall that we focus on the case where $S_{k+1}$ is (some sort of approximation of) a regularizer and hence positive definite, well-conditioned and cheap to evaluate in any given direction. It is then clear that
	we can choose $D_{k+1}$ as any symmetric positive semi-definite operator that is cheap to evaluate in all directions.
	In \cite{AMM24} we focused on $D_{k+1}=\tau_{k+1} I$, which is the most common choice for classical \LBFGS, but here we consider the more general situation that $D_{k+1}$ is a \emph{diagonal operator} with respect to a fixed orthonormal basis $(e_j)_{j\in\CI}\subset\CX$, where $\CI=\{1,2,\ldots,n\}$ if $\CX$ is $n$ dimensional and $\CI=\N$ if it is infinite dimensional.
	We say that $D_{k+1}\in\Lin(\CX)$ is diagonal with respect to $(e_j)_{j\in\CI}$ if there is a bounded sequence $(\gamma_j^{k+1})_{j\in\CI}\subset\R$ such that $D_{k+1} = \sum_{j\in\CI}\gamma_j^{k+1} e_j e_j^T$. In the following, we will usually suppress the $k$-dependency of the coefficients $(\gamma_j^{k+1})$ and write $(\gamma_j)$ instead. 
	From $\norm{D_{k+1}}=\sup_{j\in\CI}\gamma_j$ we infer that $D_{k+1}$ is bounded if and only if $\sup_{j\in\CI}\gamma_j<\infty$. Moreover, $D_{k+1}$ is positive semi-definite if and only if $\gamma_j\geq 0$ for all $j\in\CI$. 
	As $\lam(D_{k+1})=\inf_j\gamma_j$ and $\Lam(D_{k+1})=\sup_j\gamma_j$ it is easy to ensure $\lam(D_{k+1})\in T_{k+1}$ and $\Lam(D_{k+1})\in T_{k+1}$ in \cref{line_choiceofDkp1} of Algorithm~\ref{alg_ROSE}.
	If we set $\gamma_j=\tau_{k+1}$ for some $\tau_{k+1}>0$ and all $j$, we recover the classical choice $D_{k+1}=\tau_{k+1} I$. 
	Next we provide two possible choices for the diagonal elements $(\gamma_j)$ of $D_{k+1}=\sum_{j\in\CI}\gamma_j e_j e_j^T$.
	
	\paragraph{The first choice for \texorpdfstring{$D_{k+1}$}{D_{k+1}}}
	As outlined above, we want to minimize the residual $\norm{D_{k+1}s_k-z_k}$. 
	This leads to $\gamma_j=\frac{z_k^T e_j}{s_k^T e_j}$ if $s_k^T e_j\neq 0$, and $\gamma_j$ arbitrary if $s_k^T e_j=0$.
	Since we also want to ensure $\lam(D_{k+1}),\Lam(D_{k+1})\in T_{k+1}$, 
	cf. \cref{line_choiceofDkp1} in \ref{alg_ROSE}, we project each $\gamma_j$ onto 
	$T_{k+1}$. This is equivalent to saying that $(\gamma_j)$ minimizes the constrained least squares problem 
	$\norm{D_{k+1}s_k-z_k}$ s.t. $\lam(D_{k+1})\in T_{k+1} \,\wedge\, \Lam(D_{k+1})\in T_{k+1}$. 
	Writing $\hat P_{k+1}:\R\rightarrow\R$ for the projection onto $T_{k+1}$ this yields 
	\begin{equation}\label{eq_defgamBp}		
		\begin{cases*}
			\gamma_j = \hat P_{k+1}\Bigl(\frac{z_k^T e_j}{s_k^T e_j}\Bigr) & \text{ if } $s_k^T e_j \neq 0$,\\
			\gamma_j\in T_{k+1} & \text{ if } $s_k^T e_j = 0$.
		\end{cases*}
	\end{equation}
	Of course, in finite dimensions with the canonical basis $(e_j)_j$, the scalar products $z_k^T e_j$ and $s_k^T e_j$ are simply 
	the $j$-th components of the vectors $z_k$ and $s_k$, respectively. We point out that if $S_k=0$, then $y_k=z_k$, in which case \cref{eq_defgamBp} with the canonical basis is very similar to the diagonal matrix proposed in \cite{An19}. 
		
	\paragraph{The second choice for \texorpdfstring{$D_{k+1}$}{D_{k+1}}}
	Next we determine the coefficients $(\gamma_j)$ that minimize $\norm{D_{k+1}^{1/2}s_k-D_{k+1}^{-1/2}z_k}$ s.t. $\lam(D_{k+1})\in T_{k+1}\,\wedge\,\Lam(D_{k+1})\in T_{k+1}$.
	Since $\norm{D_{k+1}^{1/2}s_k-D_{k+1}^{-1/2}z_k}$ is minimal for $\gamma_j=\left\lvert\frac{z_k^T e_j}{s_k^T e_j}\right\rvert$ if $s_k^T e_j\neq 0$, and $\gamma_j$ arbitrary if $s_k^T e_j=0$, the coefficients $(\gamma_j)$ are optimal if
	\begin{equation}\label{eq_defgamBg}
		\begin{cases*}
			\gamma_j=\hat P_{k+1}\Bigl(\Bigl\lvert\frac{z_k^T e_j}{s_k^T e_j}\Bigr\rvert\Bigr) & \text{ if } $s_k^T e_j \neq 0$,\\
			\gamma_j\in T_{k+1} & \text{ if } $s_k^T e_j = 0$.
		\end{cases*}
	\end{equation}
	We have not seen this proposal for a diagonal matrix in the literature. 
	
	The following relation between the coefficients of \cref{eq_defgamBp} and \cref{eq_defgamBg} is obvious. 	
	
	\begin{lemma}\label{lem_relationbetweendifferenttaus}
		Let $(s_k,z_k)\in\CX\times\CX$ and let $(e_j)_{j\in\CI}\subset\CX$ for some $\CI\subset\N$.
		Denote by $T_{k+1}\subset\R$ a non-empty closed interval. Define $(\gamma_j)_{j\in\CI}$ according to \cref{eq_defgamBp} and $(\hat \gamma_j)_{j\in\CI}$ according to \cref{eq_defgamBg}, choosing $\gamma_j\leq\hat\gamma_j$ for any $j\in\CI$ with $s_k^T e_j=0$. Then $\gamma_j\leq\hat\gamma_j$ for all $j\in\CI$. 
	\end{lemma}
	
	\paragraph{Relationship to scaled identity}
	
	For $D_{k+1}=\tau_{k+1} I$, the following three scaling factors are particularly interesting, and they also play a role in this paper. 
	
	\begin{definition}\label{def_tau_B}
		For $(s_k,z_k)\in\CX\times\CX$ with $s_k\neq 0$ let $\rho_k:=z_k^T s_k$ and define 
		\begin{equation*}
			\tau_{k+1}^\Bp:=\frac{\rho_k}{\norm{s_k}^{2}},\qquad
			\tau_{k+1}^\GM:=\frac{\norm{z_k}}{\norm{s_k}},\qquad
			\tau_{k+1}^\Bz:=\frac{\norm{z_k}^2}{\rho_k}, 
		\end{equation*}
		where $\tau_{k+1}^\Bz$ is only defined if $\rho_k\neq 0$.
	\end{definition}
	We remark that $\tau_{k+1}^\Bp$ and $\tau_{k+1}^\Bz$ are the so-called \BB~scaling factors introduced by Barzilai and Borwein in their seminal work \cite{BB88}.
	
	It is easy to check that for $\rho_k>0$ the scaling factors 
	$\tau_{k+1}^\Bp$, $\tau_{k+1}^\GM$ and $\tau_{k+1}^\Bz$ are the minimizers of the least squares problems 
	$\norm{\tau s_k-z_k}$, $\norm{\sqrt{\tau}s_k-z_k/\sqrt{\tau}}$ and $\norm{s_k-z_k/\tau}$, respectively. 
	Correspondingly, there holds $0<\tau_{k+1}^\Bp\leq\tau_{k+1}^\GM\leq\tau_{k+1}^\Bz$ if $\rho_k>0$. 
	Note that the least squares problems are identical to those outlined for \cref{eq_Bk0Secant} if $D_{k+1}=\tau_{k+1} I$. 
	Thus, the coefficients $(\gamma_j)$ from \cref{eq_defgamBp} and \cref{eq_defgamBg} as well as 
	those introduced in \cref{def_tau_B} all realize optimal least squares fits of $s_k$ to $z_k$, but each in a different sense.
	In doing so, they all approximately realize	the secant equation $D_{k+1}s_k=z_k$. 
	The secant equation plays a fundamental role in quasi-Newton methods \cite{DS79,NW06}.
	
	It follows that the choices \cref{eq_defgamBp} and \cref{eq_defgamBg} considered for $(\gamma_j)$ in this paper
	correspond to $\tau_{k+1}^\Bp$ and $\tau_{k+1}^\GM$ in the scalar setting. 
	We emphasize that for structured \LBFGS, $\tau_{k+1}^\GM$ has emerged as the most effective choice in the scalar setting among $\tau_{k+1}^\Bp$, $\tau_{k+1}^\GM$ and $\tau_{k+1}^\Bz$, cf., e.g., \cite{JBES04,KR13,AMM24}. 
	
	Next we observe that the scalars $\tau_{k+1}^\Bp$, $\tau_{k+1}^\GM$ and $\tau_{k+1}^\Bz$ provide inner approximations for the spectrum of the \emph{average Hessian} $\overline{\nabla^2\CD_k}$. For convenience we state this only for quadratic $\CS$. The proof is similar to \cite[Section~4.1]{Gilbert1989}, hence omitted. 
	As an obvious consequence, we conclude that the spectrum of $D_{k+1}$ approximates from within the spectrum of the average Hessian 
	if we additionally impose $\tau_{k+1}^\Bp$ as lower and $\tau_{k+1}^\Bz$ as upper bound for $T_{k+1}$. 
	In \Cref{sec:experiments} we study the effect of imposing these bounds in the numerical experiments. 
		
	\begin{lemma}\label{lem_spectrumavHessDk}
		Let $\CD:\CX\rightarrow\R$ be twice continuously differentiable and $\CS:\CX\rightarrow\R$ be quadratic with Hessian $S\in\Lin(\CX)$. 
		Let $x_k,x_{k+1}\in\CX$ be such that $z_k^T s_k>0$, where
		$y_k:=\nabla\CJ(x_{k+1})-\nabla\CJ(x_k)$, $s_k:=x_{k+1}-x_k$ and $z_{k}:=y_{k}-S_{k+1}s_k$ with $S_{k+1}:=S$. 
		Let 
		\begin{equation*}
			\overline{\nabla^2\CD_k}:=\int_0^1 \nabla^2\CD(x_k+t s_k)\,\mathrm{d}t 
		\end{equation*}
		be positive semi-definite. 
		Then the scalars from \cref{def_tau_B} satisfy
		\begin{equation*}
			\lambda(\overline{\nabla^2\CD_k})
			\leq\tau_{k+1}^\Bp\leq\tau_{k+1}^\GM\leq\tau_{k+1}^\Bz
			\leq \Lambda(\overline{\nabla^2\CD_k}).
		\end{equation*}
		If, in addition, $(e_j)_{j\in\CI}\subset\CX$ is an orthonormal basis of $\CX$ and 
		we project onto $\hat T_{k+1}:=[\tau_{k+1}^\Bp,\tau_{k+1}^\Bz]$ instead of $T_{k+1}$ 
		in the formulas \cref{eq_defgamBp} or \cref{eq_defgamBg}, 
		then the diagonal operator $D_{k+1}:=\sum_{j\in\CI}\gamma_j e_j e_j^T$ with 
		$(\gamma_j)_{j\in\CI}$ according to \cref{eq_defgamBp} or \cref{eq_defgamBg} satisfies 
		\begin{equation*}
			\lambda(\overline{\nabla^2\CD_k})
			\leq\lambda(D_{k+1})\leq\Lambda(D_{k+1})\leq \Lambda(\overline{\nabla^2\CD_k}).
		\end{equation*}
	\end{lemma}	
		
	\subsection{Cautious updating}\label{sec_cautiousupdating}
	Algorithm~\ref{alg_ROSE} uses \emph{cautious updating} \cite{LiFu01} both for the decision whether $(s_k,y_k)$ is stored and for the choice of the seed matrix $B_{k+1}^{(0)}$, the latter through requiring $\lambda(D_{k+1}),\Lambda(D_{k+1})\in T_{k+1}$ in \cref{line_choiceofDkp1},
	which effectively safeguards $\norm{D_{k+1}}$ and $\norm{D_{k+1}^{-1}}$ from becoming too small or too large in relation to $\nabla\CJ(x_{k+1})$. 
	Combined, these two techniques yield sufficient control over the condition number of $B_k$ to prove, without convexity assumptions on $\CJ$, that 
	$\lim_{k\to\infty}\nabla\CJ(x_{k})=0$, cf. \cref{thm_globconv}, and that $(\CJ(x_k))$ converges q-linearly, cf. \cref{thm_rateofconvPL,thm_rateofconvstrongminimizer}. 
	It is important to note that for $\nabla\CJ(x_k)\to 0$ the lower bound $\omega_{k+1}^l$ and the upper bound $\omega_{k+1}^u$ that appear in the definition of $T_{k+1}$ satisfy $\omega_{k+1}^l\to 0$ and $\omega_{k+1}^u\to\infty$, respectively. Asymptotically, this allows $D_{k+1}$ to have arbitrarily small positive eigenvalues and arbitrarily large positive eigenvalues. This is more flexible than safeguarding with a fixed positive number which would artificially restrict the spectrum of $B_{k+1}$. 
	Cautious updating has previously been used in \LBFGS, for instance in \cite{BT20,BJRT22,LWH22,AMM24}.
	Except for our own work \cite{AMM24}, however, the cautious updating that we use in the present paper differs from that in the aforementioned references.
	Most importantly, only \cite{AMM24} proves linear convergence in a non-convex setting, and it is also the only contribution that considers cautious updating for a  \emph{structured} \LBFGS~method.

\subsection{The line search}

The step length $\alpha_k$ in classical \LBFGS~is often computed in such a way that it satisfies the weak or the strong Wolfe--Powell conditions. 
However, some authors determine $\alpha_k$ by backtracking until the Armijo condition holds. 
For \emph{structured} \LBFGS, the authors of \cite{ACG18,AMM24} observed in numerical experiments that Armijo is preferable, so we have good reason to include all these line searches.
For later reference let us make the line searches explicit. 
For Armijo with backtracking we fix constants $\beta,\sigma\in(0,1)$. The step size $\alpha_k>0$ for the iterate $x_k$ with associated descent direction~$d_k$ is obtained by successively trying for $\alpha_k$ the numbers $1,\beta,\beta^2,\beta^3,\ldots$ and accepting the first one for which $x_{k+1}=x_k+\alpha_k d_k$ satisfies
\begin{equation}\label{eq_armijocond}
	\CJ(x_{k+1}) \leq \CJ(x_k) + \alpha_k\sigma\nabla\CJ(x_k)^T d_k.
\end{equation}
For the Wolfe--Powell conditions, respectively, the strong Wolfe--Powell conditions we additionally fix $\eta\in(\sigma,1)$. 
A step size $\alpha_k>0$ is accepted if it satisfies \cref{eq_armijocond} and
\begin{equation}\label{eq_WolfePowellcond}
	\nabla\CJ(x_{k+1})^T d_k \geq \eta\nabla\CJ(x_k)^T d_k, \qquad\text{respectively,}\qquad
	\left\lvert\nabla\CJ(x_{k+1})^T d_k\right\rvert \leq \eta\left\lvert\nabla\CJ(x_k)^T d_k\right\rvert.
\end{equation}
As is common practice when working with the Wolfe--Powell conditions (weak or strong), the first trial step size for $\alpha_k$ is always the full step $\alpha_k=1$. 

\section{Convergence results}\label{sec_convana}
	
	This section presents convergence results for Algorithm~\ref{alg_ROSE}. 
	Specifically, global convergence is addressed in \cref{sec_globconv}, linear convergence in \cref{sec_linconv}, and finite convergence on suitable quadratics in \Cref{sec_finconv}.
	
	As it turns out, the convergence analysis developed in \cite{AMM24} for \AlgTu~can be generalized to cover Algorithm~\ref{alg_ROSE}.
	By doing so we essentially obtain the same convergence results for \ref{alg_ROSE} in \Cref{sec_globconv,sec_linconv} as for \Tu~in \cite{AMM24}. 
	Since the changes required in the proofs are straightforward, we only spell them out for the proof of \Cref{thm_globconv}~1).
	To distinguish the convergence properties of \Ro~from those of \Tu, we show in \Cref{sec_finconv} that for a particular class of functions, \Ro~converges after finitely many iterations, whereas \Tu~does not. 
	
	\subsection{Global convergence of Algorithm~\ref{alg_ROSE}}\label{sec_globconv}
		
	For $x_0$ in Algorithm~\ref{alg_ROSE} we define the level set, respectively, the extended level set by
	\begin{equation*}
		\Omega:=\Bigl\{x\in\CX: \; \CJ(x)\leq\CJ(x_0)\Bigr\}\qquad\text{ and }\qquad
		\Omega_\delta:=\Omega+\Ballop{\delta}(0), \enspace\text{ where } \delta>0.
	\end{equation*}
	
	The global convergence of Algorithm~\ref{alg_ROSE} holds under the following assumption. 
	
	\begin{assumption}\label{ass_globconv}	
		\phantom{to create linebreak}
		\begin{enumerate}
			\item[1)] The objective $\CJ:\CX\rightarrow\R$ is continuously differentiable and bounded below.
			\item[2)] The gradient of $\CJ$ is Lipschitz continuous in $\Omega$ with constant $L>0$, i.e., there holds
			$\norm{\nabla\CJ(x)-\nabla\CJ(\hat x)}\leq L\norm{x-\hat x}$ for all $x,\hat x\in\Omega$.
			\item[3)] The sequence $(\norm{S_k})$ in Algorithm~\ref{alg_ROSE} is bounded. 
			\item[4)] The step size $\alpha_k$ is, for all $k$, computed by Armijo with backtracking \cref{eq_armijocond} or according to the Wolfe--Powell conditions \cref{eq_WolfePowellcond}.
			In the first case, we suppose in addition that there is $\delta>0$ such that $\CJ$ or $\nabla\CJ$ is uniformly continuous in $\Omega_\delta$. 
			\item[5)] The value $c_0=0$ is only chosen in Algorithm~\ref{alg_ROSE} if 
			 with this choice there holds $\sup_k\,\norm{(B_k^{(0)})^{-1}}<\infty$
			 (which is, for instance, the case if $(\norm{S_k^{-1}})$ is bounded). 
			\item[6)] The value $C_0=\infty$ is only chosen in Algorithm~\ref{alg_ROSE} if any of the following holds:
			\begin{itemize}
				\item \Cref{line_choicetauk1} is replaced by ``Let $T_{k+1}:=[\omega_{k+1}^l,P_{k+1}(\tau_{k+1}^\GM)]$''.
				\item $\CJ$ is twice continuously differentiable, $\overline{G_k}:=\overline{\nabla^2\CJ_k}-S_k$ is symmetric positive semi-definite for all $k$, and 
				$(\norm{\overline{G_k}})$ is bounded.
				For quadratic $\CS$ and $S_k=\nabla^2\CS(x_k)$ for all $k$, we can also replace
				$\overline{G_k}$ in the preceding sentence by $\overline{\nabla^2\CD_k}:=\int_0^1 \nabla^2\CD(x_k+t s_k)\,\mathrm{d}t$.
			\end{itemize}
		\end{enumerate}		
	\end{assumption}
	
	\begin{remark}
		Parts 1)--4) of \Cref{ass_globconv} are more general than the assumptions that are typically used in the literature to analyze the convergence of \LBFGS. For instance, it is often assumed that the objective $\CJ$ is twice continuously differentiable with bounded level sets and only the Wolfe--Powell line search is discussed. 
		Let us point out two benefits of our more general setting to illustrate its usefulness. First, 
		in our numerical experience with image registration problems, Armijo with backtracking is often more effective than Wolfe--Powell, and in particular we use it in the numerical experiments of this paper. On the other hand, Wolfe--Powell is the standard line search for \LBFGS, so it is important to include it, too.
		Second, our weaker differentiability requirements cover the situation that $\CJ$ contains a penalty term, e.g. the classical quadratic penalty $\sum_i\max\{0,g_i(x)\}^2$ (here, the $g_i$ stem from inequality constraints $g_i(x)\leq 0$). Since penalty terms are frequently used in image registration \cite[Chapter~15]{KT19}, allowing penalty terms in the objective is a desirable feature.
	
		Next we comment on some aspects of 3), 5) and 6). 
		The sequence $(\norm{S_k})$ is for instance bounded if we select $S_k=\nabla^2\CS(x_k)$ for all $k$, $(x_k)$ is bounded and $\nabla^2\CS$ is Hölder continuous in $\Omega$. The sequence $(\norm{(B_k^{(0)})^{-1}})$ is for instance bounded if
		we select $S_k=\nabla^2\CS(x_k)$ for all $k$ and the regularizer $\CS$ is strongly convex.
		We stress, however, that the boundedness of $(\norm{(B_k^{(0)})^{-1}})$ is only required if we want to choose $c_0=0$ in \ref{alg_ROSE}. Yet, all convergence results hold for $c_0>0$ and the numerical results in \Cref{sec:experiments} are obtained with $c_0=10^{-6}$ and include non-convex regularizers for which a positive semi-definite approximation of the Hessian is available for $S_k$. 
		Note that $c_0=0$ implies $\omega_{k+1}^l=0$ for all $k$, while $C_0=\infty$ implies $\omega_{k+1}^u=\infty$ for all $k$. That is, for $c_0=0$, respectively, $C_0=\infty$, the lower safeguard is zero, resp., the upper safeguard is irrelevant, just as in standard \LBFGS.
		Observe in this context that if $\CS\equiv 0$ and $\CD$ is a strongly convex $C^2$ function with Lipschitz continuous gradient in $\Omega$, then \cref{ass_globconv} holds with $c_0=0$ and $C_0=\infty$, hence we can use $S_k=0$ for all $k$. In this case we recover classical~\LBFGS~if $c_s$ is smaller than the modulus of convexity of $\CD$.
	\end{remark}
	
	We now state the global convergence of \ref{alg_ROSE} in the sense $\lim_{k\to\infty}\,\norm{\nabla\CJ(x_k)} = 0$, without convexity of the objective.
	For \LBFGS-type methods, this strong form of global convergence has rarely been shown in the literature in non-convex settings, cf. the discussion in \cite{AMM24}.
		
	\begin{theorem}\label{thm_globconv}
		Let \cref{ass_globconv} hold. Then:
		\begin{enumerate}
			\item[1)] Algorithm~\ref{alg_ROSE} is well-defined. 
			\item[2)] If Algorithm~\ref{alg_ROSE} is applied with $\tl=0$, then it either terminates after finitely many iterations 
			with an $x_k$ that satisfies $\nabla\CJ(x_k)=0$ or it generates a sequence $(x_k)$ such that 
			$(\CJ(x_k))$ is strictly monotonically decreasing and convergent and there holds 
			\begin{equation*}
				\lim_{k\to\infty}\,\norm{\nabla\CJ(x_k)} = 0.
			\end{equation*}
			In particular, every cluster point of $(x_k)$ is stationary.
			\item[3)] If Algorithm~\ref{alg_ROSE} is applied with $\tl>0$, then it terminates after finitely many iterations
			with an $x_k$ that satisfies $\norm{\nabla\CJ(x_k)}\leq\tl$. 
		\end{enumerate}
	\end{theorem}
	
	\begin{proof}
		The claim 1) can be established similarly as in \cite[Lemma~4.3]{AMM24}, while 2) follows as in \cite[Theorem~4.8]{AMM24}
		and 3) is an obvious consequence of 2). 
		
		Let us spell out the two changes required in the proof of \cite[Lemma~4.3]{AMM24} to obtain 1).
		First, the intervals $[\tau^z,\tau^s]$ and $[\tau^z,\tau^g]$ that appear in the proof in
		\cite{AMM24} have to be replaced by $T_{k+1}$. 
		It is then argued in the proof that these intervals are non-empty, but this is obvious for $T_{k+1}$.
		Second, it is used that $B_k$ is positive definite, so we have to establish this here. 
		Lemma~4.1 from \cite{AMM24} yields that $B_k$ is positive definite if $B_k^{(0)}$ is.
		Since $B_k^{(0)}=D_k+S_k$ by \Cref{line_choiceofseed} of \ref{alg_ROSE}
		and since $D_k$ and $S_k$ are positive semi-definite by \Cref{line_choiceofDkp1} and \Cref{line_Skpossemdef}, 
		$B_k^{(0)}$ is positive semi-definite. To argue that it is actually positive definite, we have to distinguish two cases. 
		If the constant $c_0\geq 0$ in \ref{alg_ROSE} is positive, 
		we infer that unless $x_{k}$ is stationary (in which case the algorithm terminates before generating $B_k$, so there is nothing to show), there holds $\omega_{k}^l>0$. As $\omega_k^l$ is the lower bound of the interval $T_k$ (\cref{line_choicetauk1}), it follows from
		$\lambda(D_{k})\in T_{k}$ (\Cref{line_choiceofDkp1}) that $\lambda(D_k)\geq \omega_k^l>0$, so $D_{k}$ is positive definite, hence $B_k^{(0)}$ is, too.
		By \Cref{ass_globconv}~5) the choice $c_0=0$ is only made if $B_k^{(0)}$ is invertible, so 
		$B_k^{(0)}$ is positive definite in this case, too.
	\end{proof}
	
	\begin{remark}\label{rem_unifcontfindim}
		\phantom{lb}
			\begin{enumerate}
				\item[1)] If $(x_k)$ is bounded, the uniform continuity in \cref{ass_globconv}~4) can be dropped for finite dimensional~$\CX$. 
				\item[2)] Note that while \Cref{thm_globconv} states that cluster points of $(x_k)$ are necessarily stationary, it it does not ensure that cluster points exist. If $\CX$ is finite dimensional, then the boundedness of $(x_k)$ is sufficient for that existence. 
				If $\CX$ is infinite dimensional, then it is more delicate to ensure the existence of cluster points.
				However, if $(x_k)$ is bounded and $\nabla\CJ$ is weakly continuous, then the existence of \emph{weak} cluster points is guaranteed and it is easy to show that every weak cluster point is stationary.
			\end{enumerate}
	\end{remark}	
	
		\subsection{Rate of convergence of Algorithm~\ref{alg_ROSE}}\label{sec_linconv}
		
		The convergence rate of the classical \LBFGS~method is q-linear for the objective and r-linear for the iterates under \emph{global} strong convexity of $\CJ$, cf. \cite{LN89}, and sublinear for non-convex objectives \cite{BJRT22}. 
		For structured \LBFGS~we established linear convergence for non-convex objectives in \cite{AMM24}.
		A close inspection of the results from \cite{AMM24} reveals that they essentially apply to Algorithm~\ref{alg_ROSE}, too.
		This yields two results on the rate of convergence. 
		First we obtain under a \KL-type inequality, which is weaker than \emph{local} strong convexity, that the objective converges q-linearly and the iterates and their gradients converge r-linearly. Second, the same type of convergence also holds if there is a cluster point in whose neighborhood $\CJ$ is strongly convex, which is the classical sufficient optimality condition of second order.
		Both results rely on the following assumption.
		
		\begin{assumption}\label{ass_linconv}
			\item[1)] \Cref{ass_globconv} holds.
			\item[2)] Algorithm~\ref{alg_ROSE} is applied with $\tl=0$ and does not terminate finitely.
			\item[3)] The sequences $(\norm{B_k})$ and $(\norm{B_k^{-1}})$ are bounded.
			\item[4)] If the Armijo condition with backtracking is used for step size selection in Algorithm~\ref{alg_ROSE}, there is $\delta>0$ such that $\CJ$ is uniformly continuous in $\Omega_\delta$ or $\nabla\CJ$ is Lipschitz continuous in $\Omega_\delta$. 
		\end{assumption}

		\begin{remark}
			The boundedness assumption~3) is easy to satisfy in the structured setting of this paper. Specifically, it follows as in \cite{AMM24} that $(\norm{B_k})$ is bounded if at least one of the two statements in \cref{ass_globconv}~6) holds. Notably, the first of those statements only limits the size of the interval $T_{k+1}$ and does not involve convexity of the objective. 
			The boundedness of $(\norm{B_k^{-1}})$ is, for instance, guaranteed if $(S_k)$ is chosen uniformly positive definite. 
			If $\CS$ is strongly convex, this holds for $S_k=\nabla^2\CS(x_k)$, but 
			more sophisticated choices may be available for the problem at hand. 
			If a positive semi-definite approximation of $\nabla^2\CS(x_k)$ is available, we may choose it as $S_k$
			and, if necessary, add $\delta I$ with a small $\delta>0$ to ensure boundedness of $(\norm{B_k^{-1}})$.
			In any case, the data-fitting term $\CD$ in \cref{eq_SO} can clearly be non-convex.
		\end{remark}

		\subsubsection{Linear convergence under a \KL-type inequality}
		
		In this subsection we state the linear convergence of Algorithm~\ref{alg_ROSE} based on a \KL-type inequality.
		To introduce this inequality let us consider the sequence $(x_k)$ generated by Algorithm~\ref{alg_ROSE} and recall from \cref{thm_globconv}
		that $(\CJ(x_k))$ is strictly monotonically decreasing and that $\CJopt:=\lim_{k\to\infty}\CJ(x_k)$ exists. 
		We demand that there are $\bar k,\mu>0$ such that 
		\begin{equation}\label{eq_PLcond}
			\CJ(x_k)-\CJopt \leq \frac{1}{\mu}\norm{\nabla\CJ(x_k)}^2 \qquad \forall k\geq\bar k. 
		\end{equation}
		Comments on how this inequality relates to other \KL-type-inequalities can be found in \cite{AMM24}.
		It is not difficult to check that well-known \emph{error bound conditions} like the one in 
		\cite[Assumption~2]{TY09} imply \cref{eq_PLcond}. Thus, the following result holds in particular under any of those error bound conditions. 
		The significance of \cref{eq_PLcond} is that it allows for minimizers that are neither locally unique nor have a regular Hessian, while still resulting in linear convergence.  
		We recall that the parameter $\sigma$ appears in the Armijo~condition~\cref{eq_armijocond}.
				
		\begin{theorem}\label{thm_rateofconvPL}
			Let \cref{ass_linconv} and \cref{eq_PLcond} hold.	
			Then there exists $\xopt$ such that 
			\begin{enumerate}
				\item[1)] there hold $\nabla\CJ(\xopt)=0$ and $\CJopt=\CJ(\xopt)$;
				\item[2)] the iterates $(x_k)$ converge r-linearly to $\xopt$;
				\item[3)] the gradients $(\nabla\CJ(x_k))$ converge r-linearly to zero;
				\item[4)] the function values $(\CJ(x_k))$ converge q-linearly to $\CJ(\xopt)$. 
				Specifically, we have 
				\begin{equation*}
					\CJ(x_{k+1})-\CJ(\xopt)
					\leq \left(1-\frac{\sigma\alpha_k\mu}{\norm{B_k}}\right)\Bigl[\CJ(x_k)-\CJ(\xopt)\Bigr] \qquad \forall k\geq\bar k.
				\end{equation*}
				The supremum of the term in round brackets is strictly smaller than 1.
			\end{enumerate}
		\end{theorem}	
			
		\begin{proof}
			Identical to the proof of \cite[Theorem~4.7]{AMM24}.
	\end{proof}
	
	\begin{remark}\label{rem_boundednessBs}
		As in \cref{rem_unifcontfindim} the statements concerning $\Omega_\delta$ in \cref{ass_globconv} and in \Cref{ass_linconv} 
		can be dropped if $\CX$ is finite dimensional and $(x_k)$ is bounded. 
	\end{remark}	
	
	\subsubsection{Linear convergence under local strong convexity}	
	
	We now derive linear convergence under a different set of assumptions than in \cref{thm_rateofconvPL}. 
	For the special case $\CS\equiv 0$ and $S_k=0$ for all $k$, the following result may be viewed as an improved version of the classical convergence result \cite[Thm.~7.1]{LN89} from Liu and Nocedal on \LBFGS, the most notable improvement being that strong convexity is required only locally. 
		
	\begin{theorem}\label{thm_rateofconvstrongminimizer}
		Let \cref{ass_linconv} hold except for the statements concerning $\Omega_\delta$.
		Let $(x_k)$ have a cluster point $\xopt$ such that $\CJ\vert_{\CN}$ is $\mu$-strongly convex, where $\CN\subset\Omega$ is a convex neighborhood of $\xopt$.
		Then 
		\begin{enumerate}
			\item[1)] there holds $\CJ(\xopt) + \mu\norm{x-\xopt}^2 \leq\CJ(x)$ for all $x\in\CN$;
			\item[2)] the iterates $(x_k)$ converge r-linearly to $\xopt$;
			\item[3)] the gradients $(\nabla\CJ(x_k))$ converge r-linearly to zero;
			\item[4)] the function values $(\CJ(x_k))$ converge q-linearly to $\CJ(\xopt)$. Specifically, if 
			$\bar k$ is such that $x_k\in\CN$ for all $k\geq\bar k$, then we have
			\begin{equation}\label{eq_finalineqinproof2}
				\CJ(x_{k+1})-\CJ(\xopt)
				\leq \left(1-\frac{2\sigma\alpha_k \mu}{\norm{B_k}}\right)\Bigl[\CJ(x_k)-\CJ(\xopt)\Bigr] \qquad \forall k\geq\bar k.
			\end{equation}
			The supremum of the term in round brackets is strictly smaller than 1.
		\end{enumerate}
	\end{theorem}	

	\begin{proof}
		Identical to the proof of \cite[Theorem~4.9]{AMM24}.	
	\end{proof}

		\begin{remark}
			If $\CJ$ is $\mu$-strongly convex in the convex level set $\Omega$, then \cref{eq_finalineqinproof2} holds for $\bar k=0$. 
		\end{remark}

\subsection{Finite convergence on suitable quadratics}\label{sec_finconv}
	
In \Cref{sec_globconv,sec_linconv} we have established convergence results for \ref{alg_ROSE} for fairly general objective functions. 
We have also pointed out that these results are essentially identical to those derived in \cite{AMM24} for \Tu.
In this subsection we show in a model setting that the better Hessian approximation of \ref{alg_ROSE} results in \ref{alg_ROSE} requiring fewer iterations than \Tu. 
Specifically, we establish in this section that for certain quadratic objective functions, \ref{alg_ROSE} with $\ell=0$ can find the exact minimizer after a finite number of iterations, whereas this does not hold for \Tu. We confirm these results in the numerical experiments for quadratics in \Cref{sec:quad}.
By extension, we expect that in related settings, \ref{alg_ROSE} requires much fewer iterations than \Tu.
Indeed, for the real-world image registration problems in \Cref{sec:imgReg} we observe this to be the case. 

The key property of our model setting is that for some $k$ it enables the seed matrix $B_{k+1}^{(0)}$ of \ref{alg_ROSE} to approximate the Hessian $\nabla^2\CJ(x_{k+1})$ \emph{exactly}, yielding that $x_{k+2}$ is the \emph{exact} minimizer if $\ell=0$. Since $B_{k+1}^{(0)}=D_{k+1}+S_{k+1}$, where we choose $S_{k+1}=\nabla^2\CS(x_{k+1})$ and $D_{k+1}$ is a diagonal matrix
with diagonal elements that satisfy $\lambda(D_{k+1}),\Lambda(D_{k+1})\in T_{k+1}$, cf. \Cref{line_choiceofDkp1}, 
we can only ensure $B_{k+1}^{(0)}=\nabla^2\CJ(x_{k+1})$ if $\CD$ has a diagonal Hessian $\nabla^2\CD$ such that
$\lambda(\nabla^2\CD),\Lambda(\nabla^2\CD)\in T_{k+1}$. 
As this severely limits the class of addressable objective functions, we emphasize again that the goal of this subsection is not to prove another convergence result under general assumptions, but to show rigorously that \ref{alg_ROSE} requires significantly fewer iterations than \Tu~in some settings.

We have now motivated the majority of assumptions that are required to prove the first result.
We recall that the objective function has the form $\CJ=\CD+\CS$. 

\begin{lemma}\label{lem_goodHessianapprox}
	Let \Cref{ass_globconv} hold. 
	Let $\CD,\CS:\CX\rightarrow\R$ be convex quadratics with positive semi-definite Hessians $D,S\in\Lin(\CX)$.
	Let $(e_j)_{j\in\CI}\subset\CX$ be an orthonormal basis of $\CX$ and suppose that 
	$D$ is diagonal wrt. $(e_j)$, i.e., $D=\sum_{j\in\CI}\hat\gamma_j e_j e_j^T$ with $(\hat\gamma_j)\subset[0,\infty)$.
	Consider Algorithm~\ref{alg_ROSE} for some $k\in\N_0$ with $S_{k+1}:=S$.
	Then: 
	\begin{enumerate}
		\item[1)] \cref{eq_defgamBp} and \cref{eq_defgamBg} yield the same 
		diagonal operator $D_{k+1}=\sum_{j\in\CI}\gamma_j e_j e_j^T$ with $(\gamma_j)\subset[0,\infty)$. 
		\item[2)] For any $j$ such that $\hat\gamma_j\in T_{k+1}$ and $s_k^T e_j\neq 0$, there holds 
		$\gamma_j=\hat\gamma_j$, i.e., $B_{k+1}^{(0)} e_j=\nabla^2\CJ e_j$.
	\end{enumerate}
\end{lemma}

\begin{proof}
	\underline{Proof of 1):} 
	Comparing the formulas \cref{eq_defgamBp} and \cref{eq_defgamBg} it suffices to show that either 
	$z_k^T e_j$ and $s_k^T e_j$ have identical signs or $z_k^T e_j=0$. 
	From
	\[
		z_k = y_k - S s_k = \nabla\CD(x_{k+1})-\nabla\CD(x_{k}) = D s_k
	\]

	it follows that $z_k^T e_j = \hat\gamma_j s_k^T e_j$. Thus, either $z_k^T e_j=0$ (if $\hat\gamma_j=0$) or 
	$z_k^T e_j$ and $s_k^T e_j$ have identical signs (if $\hat\gamma_j>0$).
	
	\underline{Proof of 2):}
	The formula $z_k^T e_j = \hat\gamma_j s_k^T e_j$ derived in the proof of 1) implies that 
	$\frac{z_k^T e_j}{s_k^T e_j}=\hat\gamma_j$ for any $j$ with $s_k^T e_j\neq 0$. 
	Since $\hat\gamma_j\in T_{k+1}$,
	it follows from \cref{eq_defgamBp} that $\gamma_j=\hat\gamma_j$ for these $j$. 
\end{proof}

Observe that for $\ell=0$ the approximation property of \Cref{lem_goodHessianapprox}~2) is particularly strong in that 
$B_{k+1} e_j = \nabla^2\CJ e_j$. Since $\CJ$ is quadratic, Taylor expansion yields
\[
\nabla\CJ(x_{k+1}+d_{k+1})^T e_j = \nabla\CJ(x_{k+1})^T e_j + (d_{k+1})^T \nabla^2\CJ e_j = 
\nabla\CJ(x_{k+1})^T e_j - \nabla\CJ(x_{k+1}) e_j = 0,
\]
where the second equality relies on $B_{k+1}d_{k+1}=-\nabla\CJ(x_{k+1})$. 
This shows that if $s_k^T e_j\neq 0$ for all $j$, then $\nabla\CJ(x_{k+1}+d_{k+1}) = 0$, i.e., $x_{k+1}+d_{k+1}$ is the minimizer of $\CJ$. 
It is not difficult to show that in this case, the step size $\alpha_{k+1}=1$ is chosen and thus
\ref{alg_ROSE} terminates with $x_{k+2}$ being the minimizer. 
Essentially, we have established the following result. 

\begin{corollary}\label{cor_terminationaftertwostepsv1}
	Let \Cref{ass_globconv} hold. 
	Let $\CD,\CS:\CX\rightarrow\R$ be convex quadratics with Hessians $D,S\in\Lin(\CX)$.
	Let $(e_j)_{j\in\CI}\subset\CX$ be an orthonormal basis of $\CX$ and suppose that 
	$D$ is diagonal wrt. $(e_j)$ and positive definite. 
	Consider Algorithm~\ref{alg_ROSE} with $\ell=0$ and the choice $S_k:=S$ for all $k$. Then: 
	There is $\hat K\in\N_0$ such that if $s_K^T e_j\neq 0$ for all $j\in\CI$ and some $K\geq\hat K$, 
	then \ref{alg_ROSE} terminates for $k=K+1$ in \Cref{line_termination} with the global minimizer of $\CJ$. 
	If the constants $c_0,C_0$ in \ref{alg_ROSE} satisfy $c_0\leq\lambda(D)$ and $C_0\geq\Lambda(D)$, then $\hat K=0$.
\end{corollary}

\begin{proof}
	Our prior discussion establishes the claims, but recall that it was based on \Cref{lem_goodHessianapprox}~2) whose application requires $\lambda(D),\Lambda(D)\in T_{k+1}$. 
	The specified choice for $c_0,C_0$ guarantees that this requirement is satisfied 
	because it implies $T_{k+1}=[\omega_{k+1}^l,\omega_{k+1}^u]\supset [c_0,C_0]\supset [\lambda(D),\Lambda(D)]$.
	Here, we used that $z_k^T s_k = s_k^T D s_k > 0$ because $D$ is positive definite by assumption.
	Without knowledge about $c_0$ and $C_0$, it still holds that 
	$\omega_{k+1}^l\to 0$ and $\omega_{k+1}^u\to\infty$ for $k\to\infty$, hence 
	$\lambda(D),\Lambda(D)\in [\omega_{k+1}^l,\omega_{k+1}^u]=T_{k+1}$ holds for all sufficiently large $k$. 
	The limits for $\omega_{k+1}^l$ and $\omega_{k+1}^u$ follow from the fact that $\nabla\CJ(x_k)\to 0$, which is established in \Cref{thm_globconv}. 
\end{proof}

The assumption that $s_k^T e_j\neq 0$ for all $j$ is difficult to guarantee in general, so let us also prove a result without this assumption. 
The price to pay is that we need $S$ to be diagonal. 
This is helpful because if $B_k=D_k+S$ is diagonal, then $s_k^T e_j=0$ is equivalent to $\nabla\CJ(x_k)^T e_j=0$. Since $\nabla^2\CJ$ is also diagonal in this setting, we obtain $\nabla\CJ(x_{k+1})^T e_j = 0$ by Taylor expansion, which in turn gives $d_{k+1}^T e_j=0$. Hence, $\nabla\CJ(x_{k+1}+d_{k+1})^T e_j = \nabla\CJ(x_{k+1})^T e_j + (d_{k+1})^T \nabla^2\CJ e_j = 0$, the same result as for those indices $j$ with $s_k^T e_j\neq 0$. Necessarily, then, $\nabla\CJ(x_{k+1}+d_{k+1})^T e_j=0$ for all $j$, thus $\nabla\CJ(x_{k+1}+d_{k+1}) = 0$. As before, $\alpha_{k+1}=1$ is selected and \ref{alg_ROSE} terminates with $x_{k+2}$, which is the minimizer of $\CJ$. 
These observations yield the following result. 

\begin{lemma}\label{lem_terminationaftertwostepsv2}
	Let \Cref{ass_globconv} hold. Let $\CD,\CS:\CX\rightarrow\R$ be convex quadratics with diagonal Hessians $D,S\in\Lin(\CX)$
	wrt. to the orthonormal basis $(e_j)_{j\in\CI}$. Also, let $D$ be positive definite. 
	Consider Algorithm~\ref{alg_ROSE} with $\ell=0$ and the choice $S_k:=S$ for all $k$. 
	Let $K\in\N_0$ be the smallest number such that $\omega_K^l\leq\lambda(D)$ and $\omega_K^u\geq\Lambda(D)$ are satisfied.
	Then \ref{alg_ROSE} terminates for $k=K+1$ in \Cref{line_termination} with the global minimizer of $\CJ$. 
	If the constants $c_0,C_0$ in \ref{alg_ROSE} satisfy $c_0\leq\lambda(D)$ and $C_0\geq\Lambda(D)$, then $K=0$.
\end{lemma}

The arguments of this subsection remain valid under small perturbations of $\nabla^2\CJ$ and $B_k$, although 
the final iterate will no longer be the \emph{exact} minimizer. 
The precise statement, whose proof is omitted for brevity, reads as follows. 
We recall that $\epsilon$ is the termination tolerance in Algorithm~\ref{alg_ROSE}.

\begin{lemma}\label{lem_terminationaftertwostepsv3}
	Let \Cref{ass_globconv} hold.  
	Let $\CD:\CX\rightarrow\R$ be twice continuously differentiable and strongly convex with
	Hessian $\nabla^2\CD(x)=D+T_1(x)$ for all $x\in\CX$, where $D\in\Lin(\CX)$ is diagonal 
	wrt. to the orthonormal basis $(e_j)_{j\in\CI}$ and positive definite. 
	Let $\CS:\CX\rightarrow\R$ be twice continuously differentiable with Hessian 
	$\nabla^2\CS(x)=S+T_2(x)+T_3(x)$ for all $x\in\CX$, where $S+T_2(x)$ is positive semi-definite for all $x\in\CX$. 
	Let $\hat x\in\CX$ and $\hat\Omega:=\{x\in\CX:\CJ(x)\leq\CJ(\hat x)\}$. Consider Algorithm~\ref{alg_ROSE} with $\ell=0$ and the choice $S_k:=S+T_2(x_k)$ for all $k$. Then: For all $\epsilon>0$ there is $\delta>0$ such that if $\sup_{x\in\hat\Omega}\norm{T_1(x)}+\norm{T_2(x)}+\norm{T_3(x)}\leq\delta$, then for any $x_0\in\hat\Omega$, \ref{alg_ROSE} terminates for $k=K+1$ in \Cref{line_termination} with $x_{K+2}$ satisfying $\norm{\nabla\CJ(x_{K+2})}\leq\epsilon$,  
	where $K\in\N_0$ is the smallest number such that $\omega_K^l\leq\lambda(D)$ and $\omega_K^u\geq\Lambda(D)$ are satisfied.
	If $c_0\leq\lambda(D)$ and $C_0\geq\Lambda(D)$, then $K=0$. 
\end{lemma}

\begin{remark}\label{rem_fconvwp}
	\phantom{lb}
	\begin{enumerate}
		\item[1)] Let us specialize \Cref{lem_terminationaftertwostepsv3} to a strongly convex regularizer of the form $\CS(x)=\alpha s(x)$, $\alpha>0$, 
		$S_k:=\alpha\nabla^2 s(x_k)$ for all $k$, and a quadratic function $\CD$ with a positive definite and diagonal Hessian. 
		\Cref{lem_terminationaftertwostepsv3} with $T_1(x):=T_3(x):=S:=0$ and $T_2(x):=\alpha\nabla^2 s(x_k)$ yields that 
		as $\alpha$ decreases we expect \ref{alg_ROSE} with $\ell=0$ to terminate \emph{earlier} because $\sup_{x\in\hat\Omega}\norm{T_1(x)}+\norm{T_2(x)}+\norm{T_3(x)}$ becomes smaller. 
		In particular, if we allow $\alpha=0$ then \ref{alg_ROSE} with $\ell=0$ requires only 2 iterations if $c_0$ and $1/C_0$ are small enough, 
		and the same statement holds if $\alpha>0$ is sufficiently small. 
		In contrast, for larger values of $\alpha$ or $\ell>0$ more than two iterations may be required. 
		Similarly, if $c_0$ and $1/C_0$ are not sufficiently small or the interval $T_{k+1}$ is restricted in such a way that $\lambda(D)$ or $\Lambda(D)$ do not belong to it, then \ref{alg_ROSE} with $\ell=0$ will not terminate in two iterations.  
		The numerical results in \Cref{sec:quad} match these expectations quite well, cf. \Cref{tab:quadPrb}.
		\item[2)] Let us briefly discuss what convergence behavior we expect of \ref{alg_ROSE} for $\ell>0$. 
		For $\ell=0$ the key point is that a good approximation $B_{k+1}^{(0)}$ of $\nabla^2\CJ(x_{k+1})$ translates into $B_{k+1}$ being a good approximation of $\nabla^2\CJ(x_{k+1})$. 
		For $\ell>0$ the identity $B_{k+1}^{(0)}=B_{k+1}$ is no longer true, so although $B_{k+1}^{(0)}=\nabla^2\CJ(x_{k+1})$ remains valid, 
		the rank-two updates actually destroy the perfect approximation and guarantee that $B_{k+1}\neq\nabla^2\CJ(x_{k+1})$, unless they cancel each other out.
		On the other hand, as $\ell$ increases \ref{alg_ROSE} becomes more similar to a BFGS-type method, so we expect a tendency of improving (linear) convergence rates, which suggests lower iteration numbers. 
		By observing that the rank-two updates modify the eigenvalues of $B_{k+1}$ vs. $B_{k+1}^{(0)}$, we may further suspect
		that the updates are most helpful in settings where $B_{k+1}^{(0)}$ is unable to approximate the spectrum of $\nabla^2\CJ$ well, for instance if 
		$T_{k+1}$ is too small. Thus, if $T_{k+1}$ does not include $\lambda(D)$ or $\Lambda(D)$ (by a certain margin), the performance of 
		\ref{alg_ROSE} with $\ell=0$ should be worse than with $\ell>0$, regardless of the value of $\alpha$. 
		Again, the numerical results in \Cref{sec:quad} are very well-aligned with these expectations, cf. \Cref{tab:quadPrb}.  
		\item[3)] It is clear that if $\nabla^2\CD$
		is diagonal but not a scalar multiple of the identity, then 
		the seed matrix of \Tu~cannot agree with the exact Hessian. Therefore, the iteration numbers of \Tu~may be significantly larger than those of \ref{alg_ROSE} in the setting of \Cref{lem_terminationaftertwostepsv2}. \Cref{tab:quadPrb} confirms this.
		By extension, we also expect lower iteration numbers from \Ro~if $\CD$ has a Hessian that is close to a diagonal matrix (at least near the minimizer)
		or ill-conditioned, cf. also \Cref{lem_terminationaftertwostepsv3}. The results of the real-world experiments in \Cref{sec_rlprobs} support this view. 
	\end{enumerate}	
\end{remark}	
		
%%%%%%%%%%%%%%%%%%%%%%%%%%%%%%%%%%%%%%%%%%%%%%%%%%%%%%%%%%%%%%%%%%%%%%%%%%%%%%%%%%%%%%%%%%%%%%%%%%%%%%%%%%%%%%%%%%%%%%%%%%
%%%%%%%%%%%%%%%%%%%%%%%%%%%%%%%%%%%%%%%%%%%%%%%%%%%%%%%%%%%%%%%%%%%%%%%%%%%%%%%%%%%%%%%%%%%%%%%%%%%%%%%%%%%%%%%%%%%%%%%%%%

\newcommand{\Ds}{{\mathrm{Ds}}}
\newcommand{\Dg}{{\mathrm{Dg}}}

\section{Numerical experiments}\label{sec:experiments}
						
		The numerical experiments are divided into two parts. 
		First we consider real-life image registration problems to demonstrate the practical merits of \ref{alg_ROSE}, then we study quadratic model problems to illustrate its convergence properties in an academic setting. 
		All experiments are performed in \textsc{MATLAB} (R2023b) on an Apple M1 Pro with 32GB of RAM.
		
		\subsection{Real-life image registration problems}\label{sec_rlprobs}
		
		We first demonstrate the effectiveness of Algorithm~\ref{alg_ROSE} over its predecessor \Tu~from \cite{AMM24} using the same 22 real-life large-scale highly non-convex and ill-posed image registration problems that were used in \cite{AMM24} to compare \Tu~with \LBFGS~and other structured \LBFGS~methods. 
		A description of these problems is provided next. 
		
		\subsubsection{Problems under consideration}\label{sec:imgReg}
		
		Registration problems are generally highly non-convex and ill-posed. Given a
		pair of images $T$ and $R$, the goal is to find a transformation field $\phi$
		such that the transformed image $T(\phi)$ is similar to $R$, i.e., 
		$T\circ\phi\approx R$. To determine $\phi$, we solve an unconstrained optimization problem
		\begin{equation}\label{eq_basicimageregprob}
			\min_\phi \CJ(\phi) = \CD(\phi;T,R)+ \alpha S(\phi),
		\end{equation}
		where $\CD$ measures the similarity between the transformed image $T(\phi)$ and
		$R$.
		The regularizer $\CS=\alpha S$ guarantees that the problem is solvable and it enforces smoothness in the field.
		We follow the ``discretize-then optimize'' approach. 
		To obtain a clearer comparison we work with a fixed discretization instead of a multilevel approach.
		In consequence, for each problem the transformation field satisfies $\phi \in \R^n$ with a fixed $n$ that defines the number of variables. 
		The values of $n$ are provided in \cref{tab:IRproblems}.
		
		An important quality measure for the registration is the \emph{target registration error (TRE)}. It measures the distance between the landmark locations in the reference image and in the target image transformed with the transformation field $\phi$.
		
		The 22 test cases, listed in \cref{tab:IRproblems}, cover many different registration models, estimating small to large deformations. The three-dimensional (3D) lung CT images are from the well-known DIR dataset \cite{Knig2018,Castillo2009}, 
		and the rest of the datasets are from \cite{FAIR09}; in \cref{fig_imgreg} we display five of the datasets together with registration results. 2D-Disc images are an academic example with large deformations from \cite{BMR13}.			
		Let us also point out that the test cases do not include landmark constrained registration problems, but that 
		these problems fit in the framework of this paper if the constraints are eliminated as proposed in \cite{HHM09}, resulting in an unconstrained optimization problem formulated in the basis of the constraint set.
		
		The test cases comprise the fidelity measures sum of squared
		difference (SSD), mutual information (MI) \cite{V95} and normalized gradient fields (NGF) \cite{Haber2006}. They include 
		a quadratic first order (Elas)~\cite{Broit1981}, a quadratic second order (Curv)~\cite{Fischer2004} and a non-quadratic first order (H-elas) \cite{BMR13} regularizer. We stress that the hyperelastic regularizer (H-elas) is non-convex. 
		Also, computing the full Hessian of the hyperelastic regularizer is expensive. As proposed in \cite{BMR13} we therefore use a Gauss--Newton-like approximation of the Hessian for $S_k$. In particular, this approximation is positive semi-definite. 
		It is important to note here that the convergence analysis of \ref{alg_ROSE} allows for $S_k\neq\nabla^2\CS(x_k)$.	
		
		%----------------------------------------------------------------------------
\begin{table}[ht]
\footnotesize
\caption{We use 22~non-quadratic image registration problems as test cases (TC). 
%for the performance evaluation of the proposed optimization strategies. 
Data-fidelity (MI, NGF, SSD) and regularization (Curvature, Elasticity, Hyperelasticity) are
denoted by $\CD$ and $\CS$, respectively; the regularization parameter is
$\alpha$, cf. \cref{eq_basicimageregprob}. 
The problem size is $n=d\cdot\prod_{k=1}^dm_k$, where $d\in\{2,3\}$ denotes
the data dimensionality (2D or 3D) and $m_k$ the corresponding data resolution. 
The data resolution for Hands, PET-CT and MRI data are $128 \times 128$, for Disc data $16\times16$, Lung data $64 \times 64 \times X$ where $X \in [24,28]$, and Brain data $32 \times 16 \times 32$.
The last column (Initial TRE) reports the initial target registration error and in brackets the standard deviation, cf.~\cref{sec:imgReg}.
In \cref{fig_imgreg} we display five of the data sets together with registration results.
}\label{tab:IRproblems}

\centering
\begin{tabular}{rlrllrc}
	\toprule
	TC
	& Dataset & $n$
	& $\CD$ & $\CS$ & $\alpha$
	& Initial TRE
	\\
	\midrule
	1 & 2D-Hands & $32\,768$	 & SSD & Curv   & $1.5 \cdot 10^3$ & 1.04 (0.62)\\
	2 & 2D-Hands & $32\,768$	 & SSD & Elas   & $1.5 \cdot 10^3$ & 1.04 (0.62)\\
	3 & 2D-Hands &	$32\,768$    & SSD & H-elas & $(10^3,20)$ & 1.04 (0.62)\\
	4 & 2D-Hands & $32\,768$	 & NGF  & Curv   & $0.01$ & 1.04 (0.62)\\
	5 & 2D-Hands & $32\,768$	 & NGF  & Elas   & $1$  & 1.04 (0.62)\\
	6 & 2D-Hands &	$32\,768$    & NGF & H-elas & $(1,1)$ & 1.04 (0.62)\\
	7 & 2D-Hands & $32\,768$	 & MI  & Curv   & $5 \cdot 10^{-3}$ & 1.04 (0.62)\\
	8 & 2D-Hands & $32\,768$	 & MI  & Elas   & $5 \cdot 10^{-3}$  & 1.04 (0.62)\\
	9 & 2D-Hands &	$32\,768$    & MI & H-elas & $(10^{-3},1)$ & 1.04 (0.62)\\    
	\midrule
	10 & 2D-PET-CT & $32\,768$	 & MI  & Elas   & $10^{-4}$ & N.A.\\
	11 & 2D-PET-CT & $32\,768$	 & MI  & Curv   & $0.1$ & N.A.\\
	12 & 2D-PET-CT & $32\,768$	 & NGF  & Elas   & $0.05$ & N.A.\\
	13 & 2D-PET-CT & $32\,768$	 & NGF  & Curv   & $10$ & N.A.\\
    \midrule
    14 & 2D-MRI-Head & $32\,768$	 & MI  & Elas   & $10^{-3}$ & N.A.\\
	15 & 2D-MRI-Head & $32\,768$	 & NGF  & Elas   & $0.1$ & N.A.\\	
	\midrule
	16 & 2D-Disc  & $512$	 & SSD & H-elas & $(100,20)$ & N.A. \\
	\midrule
	17 & 3D-Brain  & $49\,152$	 & SSD & H-elas & $(100,10,100)$ & N.A. \\
	\midrule
	18 & 3D-Lung  & $294\,912$ & NGF & Curv   & $100$ & 3.89 (2.78) \\
    19 & 3D-Lung  & $307\,200$ & NGF & Curv   & $100$ & 9.83 (4.86) \\
    20 & 3D-Lung  & $319\,488$ & NGF & Curv   & $100$ & 6.94 (4.05) \\
    21 & 3D-Lung  & $331\,776$ & NGF & Curv   & $100$ & 7.48 (5.51) \\
    22 & 3D-Lung  & $344\,064$ & NGF & Curv   & $100$ & 4.34 (3.90)\\
	\bottomrule    
\end{tabular}
\end{table}

	\subsubsection{Algorithmic settings}
	
	For \Tu~and \ref{alg_ROSE} we use $S_k=\nabla^2\CS(x_k)$, except for the hyperelastic regularizer where $S_k$ only approximates $\nabla^2\CS(x_k)$. 
	The seed matrix in \Tu~is $B_k^{(0)}=\tau_k I + S_k$, and we compute $\tau_k$ by the adaptive approach from \cite{AMM24}, which is the best performing method of \cite{AMM24}. Since we use only one approach for $\tau_k$, we will refer to it simply as \Tu. 
	In \ref{alg_ROSE} we have $B_k^{(0)}=D_k + S_k$, and we determine the diagonal entries of $D_k$ according to \cref{eq_defgamBp}, respectively, \cref{eq_defgamBg}, which we refer to as \Ro-$\Ds$ and \Ro-$\Dg$, respectively. 
		The linear system in the two-loop recursion %of \ref{alg_ROSE} and \Tu~
		is solved inexactly and matrix free using \minres~\cite{Paige1975} with a Jacobi preconditioner. Unless stated otherwise, \minres~is terminated after 50 iterations or when the relative residual falls below $10^{-2}$. For the 22 image registration problems under consideration, these settings were identified in \cite{AMM24} to work well for \Tu~and to outperform the preconditioned conjugate gradients method. 
			
		The image processing operations are carried out matrix free with the open-source image registration toolbox \FAIRTEXT~\cite{FAIR09}. 
		For the stopping criteria of the optimization methods we follow \cite[p.~78]{FAIR09}.
		That is, we stop if all of the conditions
		\begin{itemize}
			\item $\lvert\CJ(x_k) - \CJ(x_{k-1})\rvert \le 10^{-5}\bigl(1 + \lvert\CJ(x_0)\rvert\bigr)$, 
			\item $\|x_k - x_{k-1}\|\le 10^{-3}\bigl(1 + \|x_k\|\bigr)$,
			\item $\|\nabla\CJ(x_k)\| \leq 10^{-3}\bigl(1 + \lvert\CJ(x_0)\rvert\bigr)$
		\end{itemize}
		are satisfied. 
		We consistently use $\ell = 5$ in \ref{alg_ROSE} and \Tu. 
		We employ the Armijo line search routine from FAIR \cite{FAIR09} with parameters $LSmaxIter = 50$ and $LSreduction=10^{-4}$, where the latter corresponds to $\sigma$ in \cref{eq_armijocond}. 
		We do not consider the Wolfe--Powell line search because it does not work as well on image registration problems \cite{AMM24}.
		The remaining parameter values of Algorithms~\ref{alg_ROSE} and \Tu~are specified in \cref{tab_paramvalues}, which also
		contains parameter values for Algorithm~\ref{alg_es} that we introduce below.
				
		\begin{table}
			\footnotesize
			\caption{Parameter values for Algorithms~\ref{alg_ROSE} and \ref{alg_es}}
			\label{tab_paramvalues}
			\centering
		\begin{tabular}{c|c|c|c|c|c|c|c|c|c}
			\toprule
			\multicolumn{5}{c|}{Algorithm~\ref{alg_ROSE}} & \multicolumn{5}{c}{Algorithm~\ref{alg_es}} \\
			\midrule
			$c_s$ & $c_0$ & $C_0$ & $c_1$ & $c_2$ & $\eps_0$ & $\eps_1$ & $\eta_0$ & $\eta_1$ & $\eta_2$\\
			\midrule
			$10^{-9}$ & $10^{-6}$ & $10^6$ & $10^{-6}$ & $1$ & $10^{-3}$ & $10^{-4}$ & $10$ & $30$ & $50$\\
			\bottomrule
		\end{tabular}
		\end{table}
		
		\subsubsection{Evaluation measures}	
		
		We select run-time and solution accuracy as the main criteria to evaluate the performance of the algorithms, where the solution accuracy is measured with target registration error  (TRE) \cite{Fitzpatrick2001} discussed in \Cref{sec:imgReg}. 
		
		To visualize the performance we use the performance profiles of Dolan and Moré \cite{Dolan2002} which allow to compare the performance of several optimization methods on a given set of problems with respect to a performance metric (e.g., run-time). Denote by $S$ the set of methods, by $P$ the set of problems, and by $t_{p,s}\in(0,\infty]$ the value that method $s\in S$ achieves on problem $p\in P$ in the performance metric, where a smaller value of $t_{p,s}$ is better and $t_{p,s}=\infty$ indicates that algorithm $s$ did not solve $p$. The performance profile for $s$ is the function $\rho_s:[1,\infty)\rightarrow [0,1]$ given by 
			\begin{equation*}
				\rho_{s}(\tau) = \frac{\left\lvert\bigl\{p\in P: \, r_{p,s} \leq \tau\bigr\}\right\rvert}{\lvert P\rvert}, \qquad \text{where } r_{p,s} = \frac{t_{p,s}}{\min\bigl\{t_{p,\sigma}: \, \sigma \in S\bigr\}}.
			\end{equation*}
			We see that $\rho_s$ is the cumulative distribution function with respect to the performance metric $t$. 
			Note that $\tau$ in $\rho_{s}(\tau)$ is not related to the scaling factor $\tau$ used for seed matrices, but both are standard notation.
		
		\subsubsection{Results}
		
		\paragraph{Experimental comparison of different choices for \texorpdfstring{$\mathbf{D_k}$}{D_k}}
		
		\begin{table}
			\footnotesize
			\caption{Diagonal approximation schemes and choice of bounds $a,b$ for $\hat T=[a,b]\cap T$. 
				We work with the absolute values of $\tau^\Bp$ and $\tau^\Bz$ to account for the case $s^T z<0$. 
				It is easy to see that $|\tau^\Bp| \leq |\tau^\Bz|$. 
				We point out that the parameters for the cautious updates are chosen such that $\omega^l \ll |\tau^\Bp|$ and $\omega^u \gg |\tau^\Bz|$.
				}
			\label{tab_diag_bounds}
			\centering
			\begin{tabular}{c|c|c|c}
				\toprule
				No. &
				diagonal approximation &
				lower bound $a$&
				upper bound $b$\\
				\midrule
				1. & \multirow{3}{*}{\Ro-$\Ds$: Formula \cref{eq_defgamBp}} 
				& $\omega^l$      & $\omega^u$ \\
				2. & 					  									     & $\omega^l$       &  $\min(|\tau^\Bz|, \omega^u)$\\
				3. & 					  									     & $\max(|\tau^\Bp|, \omega^l)$\ 				&  $\min(|\tau^\Bz|, \omega^u)$\ \\
				\midrule
				4. & \multirow{3}{*}{\Ro-$\Dg$: Formula \cref{eq_defgamBg}} 
				& $\omega^l$      & $\omega^u$ \\
				5. & 					  																  & $\omega^l$       &  $\min(|\tau^\Bz|, \omega^u)$\\
				6. & 					  																  & $\max(|\tau^\Bp|, \omega^l)$\ 				&  $\min(|\tau^\Bz|, \omega^u)$\ \\
				\bottomrule
			\end{tabular}
		\end{table}
	
		\pgfplotsset{myaxisstyle/.style={
		axis y line=left,
		scale = 0.55,
		title style={at={(0.8,1.7)}},
		axis lines = left,
		xlabel = $\tau$,
		ymin=0, ymax=1.1,
		ylabel style={at={(-0.3,1)}},
		xlabel style={at={(1,-0.3)}},
		minor tick num=2,
		grid=both,
		grid style={line width=.1pt, draw=gray!05},
		cycle list name=color list,
		line width=0.8pt,
		x tick label style={
			/pgf/number format/.cd,
			fixed,
			precision=3,
			/tikz/.cd
		},
}}

\begin{figure*}[ht!]
    \centering
    \begin{minipage}[b]{\linewidth}
    	 \centering
    	(i) \Ro-$\Ds$ vs. \Tu
    \end{minipage}
    \begin{minipage}[b]{0.32\linewidth}
    \pgfplotstableread[col sep=comma]{data/Ds_time.csv}\loadedtable
    \begin{tikzpicture}[font=\footnotesize,inner sep=2pt, outer sep=0pt]
      \begin{axis}[
        myaxisstyle,  
        title = {(a) run-time},
        ylabel = $\rho(\tau)$,
        xmin=1, xmax=5,
      ]
      \addplot[red] table[x=x, y=c1]{\loadedtable};
      \addplot[violet] table[x=x, y=c2]{\loadedtable};
      %\addplot[cyan] table[x=x, y=c5]{\loadedtable};
      \addplot[teal] table[x=x, y=c3]{\loadedtable};
      %\addplot[blue] table[x=x, y=c9]{\loadedtable};
      \addplot[cyan] table[x=x, y=c4]{\loadedtable};
      \end{axis}
    \end{tikzpicture}
    \end{minipage}
    \hfill
    \begin{minipage}[b]{0.32\linewidth}
   		\pgfplotstableread[col sep=comma]{data/Ds_mTRE.csv}\loadedtable
    	\begin{tikzpicture}[font=\footnotesize,inner sep=2pt, outer sep=0pt]
    		\begin{axis}[
    			myaxisstyle,  
    			title = {(b) accuracy},
    			xmin=1, xmax=1.075,
    			]
				\addplot[red] table[x=x, y=c1]{\loadedtable};
				\addplot[violet] table[x=x, y=c2]{\loadedtable};
				%\addplot[cyan] table[x=x, y=c5]{\loadedtable};
				\addplot[teal] table[x=x, y=c3]{\loadedtable};
				%\addplot[blue] table[x=x, y=c9]{\loadedtable};
				 \addplot[cyan] table[x=x, y=c4]{\loadedtable};
    		\end{axis}
    	\end{tikzpicture}
    \end{minipage}
    \hfill
    \begin{minipage}[b]{0.32\linewidth}
	\pgfplotstableread[col sep=comma]{data/Ds_red.csv}\loadedtable
	\begin{tikzpicture}[font=\footnotesize,inner sep=2pt, outer sep=0pt]
	\begin{axis}[
		myaxisstyle,  
		title = {(c) objective value reduction},
		xmin=1, xmax=1.1,
		]
		\addplot[red] table[x=x, y=c1]{\loadedtable};
		\addplot[violet] table[x=x, y=c2]{\loadedtable};
		%\addplot[cyan] table[x=x, y=c5]{\loadedtable};
		\addplot[teal] table[x=x, y=c3]{\loadedtable};
		%\addplot[blue] table[x=x, y=c9]{\loadedtable};
		 \addplot[cyan] table[x=x, y=c4]{\loadedtable};
	\end{axis}
	\end{tikzpicture}
    \end{minipage}
%%%%%%%%%% Dg %%%%%%%%%%%%%%%%%%%%
    \begin{minipage}[b]{\linewidth}
    	\centering
		(ii) \Ro-$\Dg$ vs. \Tu
	\end{minipage}
	\begin{minipage}[b]{0.32\linewidth}
		\pgfplotstableread[col sep=comma]{data/Dg_time.csv}\loadedtable
		\begin{tikzpicture}[font=\footnotesize,inner sep=2pt, outer sep=0pt]
		\begin{axis}[
		myaxisstyle,  
		title = {(a) run-time},
		ylabel = $\rho(\tau)$,
		xmin=1, xmax=5,
		]
		\addplot[red] table[x=x, y=c1]{\loadedtable};
		\addplot[violet] table[x=x, y=c2]{\loadedtable};
		%\addplot[cyan] table[x=x, y=c6]{\loadedtable};
		\addplot[teal] table[x=x, y=c3]{\loadedtable};
		%\addplot[blue] table[x=x, y=c10]{\loadedtable};
		 \addplot[cyan] table[x=x, y=c4]{\loadedtable};
		\end{axis}
		\end{tikzpicture}
	\end{minipage}
	\hfill
	\begin{minipage}[b]{0.32\linewidth}
		\pgfplotstableread[col sep=comma]{data/Dg_mTRE.csv}\loadedtable
		\begin{tikzpicture}[font=\footnotesize,inner sep=2pt, outer sep=0pt]
		\begin{axis}[
		myaxisstyle,  
		title = {(b) accuracy},
		xmin=1, xmax=1.075,
		]
		\addplot[red] table[x=x, y=c1]{\loadedtable};
		\addplot[violet] table[x=x, y=c2]{\loadedtable};
		%\addplot[cyan] table[x=x, y=c6]{\loadedtable};
		\addplot[teal] table[x=x, y=c3]{\loadedtable};
		%\addplot[blue] table[x=x, y=c10]{\loadedtable};
		 \addplot[cyan] table[x=x, y=c4]{\loadedtable};
		\end{axis}
		\end{tikzpicture}
	\end{minipage}
	\hfill
	\begin{minipage}[b]{0.32\linewidth}
		\pgfplotstableread[col sep=comma]{data/Dg_red.csv}\loadedtable
		\begin{tikzpicture}[font=\footnotesize,inner sep=2pt, outer sep=0pt]
		\begin{axis}[
		myaxisstyle,  
		title = {(c) objective value reduction},
		xmin=1, xmax=1.1,
		]
		\addplot[red] table[x=x, y=c1]{\loadedtable};
		\addplot[violet] table[x=x, y=c2]{\loadedtable};
		%\addplot[cyan] table[x=x, y=c6]{\loadedtable};
		\addplot[teal] table[x=x, y=c3]{\loadedtable};
		%\addplot[blue] table[x=x, y=c10]{\loadedtable};
		 \addplot[cyan] table[x=x, y=c4]{\loadedtable};
		\end{axis}
		\end{tikzpicture}
	\end{minipage}
    \begin{minipage}[b]{\linewidth}
   	\centering
   	\begin{tikzpicture}
   	\draw[red,solid] (0.5,0) -- (1.0,0) node[right] {$[\omega^l, \omega^u]$} ;
   	\draw[violet] (3.0,0) -- (3.5,0) node[right] {$[\omega^l, |\tau^z|]$};
   	\draw[teal] (5.5,0) -- (6.0,0) node[right] {$[|\tau^s|, |\tau^z|]$};
   	\draw[cyan] (8.0,0) -- (8.5,0) node[right] {\Tu} ;
   	\end{tikzpicture}
   \end{minipage}
    \caption{Performance profiles for different variants of \Ro~compared with \Tu. %~with diagonal matrix, respectively, scaled identity for $D_k$. 
    	%The top and bottom row consider \Ro with the diagonal formulas \cref{eq_defgamBg} and \cref{eq_defgamBp}, respectively. 
    	Each variant of \Ro~is combined with three different choices of $\hat T_{k+1}$. %and compared against \Ro.lower and upper bound for 
    }
    \label{fig:perf_diag}
\end{figure*}

		We pair the two schemes \Ro-$\Ds$ and \Ro-$\Dg$ 
		with various restrictions of $T_{k+1}$. 
		Here, by restriction we mean that we choose $\lambda(D_{k+1}),\Lambda(D_{k+1})$ in \Cref{line_choiceofDkp1} of \ref{alg_ROSE} from a sub-interval $\hat T_{k+1}:=[a_{k+1},b_{k+1}]\cap T_{k+1}$ 
		(this is compatible with the algorithm since it still ensures $\lambda(D_{k+1}),\Lambda(D_{k+1})\in T_{k+1}$). 
		Specifically, we are interested in sub-intervals that use $\tau_{k+1}^\Bp$ as lower and $\tau_{k+1}^\Bz$ as upper bound 
		because this guarantees that the spectrum of $D_{k+1}$ is related to that of the (average) Hessian of the data-fidelity term, cf.~\cref{lem_spectrumavHessDk}. 
		Omitting the index $k+1$, we detail the different combinations of lower and upper bounds that we use for $\hat T_{k+1}$ in \cref{tab_diag_bounds}. 
		
		\Cref{fig:perf_diag} shows that all variants of \ref{alg_ROSE} are either faster or at least similar to \Tu. 
		All variants of \Ro-$\Dg$ outperform the adaptive version of \Tu, while only one variant of \Ro-$\Ds$ is faster than \Tu.  
		
		In both \Ro-$\Ds$ and \Ro-$\Dg$, the variant with lower bound $|\tau^\Bp|$ is the fastest but yields the lowest accuracy. Less run-time is attributable to lower objective value reduction. The other two variants with lower bound $\omega^l$ are almost identical in performance, but the one with upper bound $|\tau^\Bz|$ has a slight advantage in terms of accuracy and run-time, which is why we focus on improving this variant's run-time by managing the accuracy of the linear solver over the iterations.
						
		\paragraph{Effect of linear solver on run-time performance}
	
		\pgfplotsset{myaxisstyle/.style={
		axis y line=left,
		scale = 0.55,
		title style={at={(0.8,1.7)}},
		axis lines = left,
		xlabel = $\tau$,
		ymin=0, ymax=1.1,
		ylabel style={at={(-0.3,1)}},
		xlabel style={at={(1,-0.3)}},
		minor tick num=2,
		grid=both,
		grid style={line width=.1pt, draw=gray!05},
		cycle list name=color list,
		line width=0.8pt,
		x tick label style={
			/pgf/number format/.cd,
			fixed,
			precision=3,
			/tikz/.cd
		},
}}

\begin{figure*}[ht!]
    \centering
    \begin{minipage}[b]{0.32\linewidth}
    \pgfplotstableread[col sep=comma]{data/DsDg_time.csv}\loadedtable
    \begin{tikzpicture}[font=\footnotesize,inner sep=2pt, outer sep=0pt]
      \begin{axis}[
        myaxisstyle,  
        title = {(a) run-time},
        ylabel = $\rho(\tau)$,
        xmin=1, xmax=4,
      ]
      \addplot[red] table[x=x, y=c1]{\loadedtable};
      \addplot[violet] table[x=x, y=c2]{\loadedtable};
      \addplot[cyan] table[x=x, y=c3]{\loadedtable};
      %\addplot[blue] table[x=x, y=c9]{\loadedtable};
      %\addplot[green] table[x=x, y=c14]{\loadedtable};
      \end{axis}
    \end{tikzpicture}
    \end{minipage}
    \hfill
    \begin{minipage}[b]{0.32\linewidth}
   		\pgfplotstableread[col sep=comma]{data/DsDg_feval.csv}\loadedtable
    	\begin{tikzpicture}[font=\footnotesize,inner sep=2pt, outer sep=0pt]
    		\begin{axis}[
    			myaxisstyle,  
    			title = {(b) function evaluations},
    			xmin=1, xmax=4,
    			]
				\addplot[red] table[x=x, y=c1]{\loadedtable};
				\addplot[violet] table[x=x, y=c2]{\loadedtable};
				\addplot[cyan] table[x=x, y=c3]{\loadedtable};
				%\addplot[teal] table[x=x, y=c7]{\loadedtable};
				%\addplot[blue] table[x=x, y=c9]{\loadedtable};
				 %\addplot[green] table[x=x, y=c14]{\loadedtable};
    		\end{axis}
    	\end{tikzpicture}
    \end{minipage}
    \hfill
    \begin{minipage}[b]{0.32\linewidth}
	\pgfplotstableread[col sep=comma]{data/DsDg_minres.csv}\loadedtable
	\begin{tikzpicture}[font=\footnotesize,inner sep=2pt, outer sep=0pt]
	\begin{axis}[
		myaxisstyle,  
		title = {(c) MINRES iterations},
		xmin=1, xmax=5,
		]
		\addplot[red] table[x=x, y=c1]{\loadedtable};
		\addplot[violet] table[x=x, y=c2]{\loadedtable};
		\addplot[cyan] table[x=x, y=c3]{\loadedtable};
		%\addplot[teal] table[x=x, y=c7]{\loadedtable};
		%\addplot[blue] table[x=x, y=c9]{\loadedtable};
		 %\addplot[green] table[x=x, y=c14]{\loadedtable};
	\end{axis}
	\end{tikzpicture}
    \end{minipage}
    \begin{minipage}[b]{\linewidth}
   	\centering
   	\begin{tikzpicture}
   	\draw[red] (0.5,0) -- (1.0,0) node[right] {\Ro-$\Ds$: $[\omega^l, |\tau^z|]$};
   	\draw[violet] (5,0) -- (5.5,0) node[right] {\Ro-$\Dg$: $[\omega^l, |\tau^z|]$};
   	\draw[cyan] (9.5,0) -- (10,0) node[right] {\Tu};
   	\end{tikzpicture}
   \end{minipage}
    \caption{Performance profiles comparing the best variant of \Ro-$\Ds$ and \Ro-$\Dg$ to \Tu.
    }
    \label{fig:perf_minres}
\end{figure*}

		While a diagonal choice of $D_k$ infuses more information into the structured \LBFGS~method compared to a scaled identity, the run-time is below the anticipated level. \Cref{fig:perf_minres}~(b) confirms that the diagonal schemes require fewer function evaluations than the scaled identity, but indicates that for the diagonal choice the linear solver requires a much higher number of total iterations. For instance, it requires more than double the amount of iterations for $70\%$ of the problems. This is the main reason for its underperformance.
		
		Since \Ro-$\Dg$ is superior to \Ro-$\Ds$, 
		which is consistent with the literature on structured \LBFGS~\cite{JBES04,KR13,AMM24},
		we exclude \Ro-$\Ds$ from the remaining experiments on image registration.
		
		\paragraph{Improved run-time performance with earlier stopping}
		
	    \pgfplotsset{myaxisstyle/.style={
		axis y line=left,
		scale = 0.55,
		title style={at={(0.8,1.7)}},
		axis lines = left,
		xlabel = $\tau$,
		ymin=0, ymax=1.1,
		ylabel style={at={(-0.3,1)}},
		xlabel style={at={(1,-0.3)}},
		minor tick num=2,
		grid=both,
		grid style={line width=.1pt, draw=gray!05},
		cycle list name=color list,
		line width=0.8pt,
		x tick label style={
			/pgf/number format/.cd,
			fixed,
			precision=3,
			/tikz/.cd
		},
}}

\begin{figure*}[ht!]
    \centering
    \begin{minipage}[b]{0.32\linewidth}
    \pgfplotstableread[col sep=comma]{data/MR_DsDg_time.csv}\loadedtable
    \begin{tikzpicture}[font=\footnotesize,inner sep=2pt, outer sep=0pt]
      \begin{axis}[
        myaxisstyle,  
        title = {(a) run-time},
        ylabel = $\rho(\tau)$,
        xmin=1, xmax=4,
      ]
      \addplot[violet] table[x=x, y=c1]{\loadedtable};
      \addplot[orange] table[x=x, y=c2]{\loadedtable};
      \addplot[cyan] table[x=x, y=c3]{\loadedtable};
      \addplot[teal] table[x=x, y=c4]{\loadedtable};
      %\addplot[blue] table[x=x, y=c9]{\loadedtable};
      %\addplot[green] table[x=x, y=c14]{\loadedtable};
      \end{axis}
    \end{tikzpicture}
    \end{minipage}
    \hfill
    \begin{minipage}[b]{0.32\linewidth}
   		\pgfplotstableread[col sep=comma]{data/MR_DsDg_mTRE.csv}\loadedtable
    	\begin{tikzpicture}[font=\footnotesize,inner sep=2pt, outer sep=0pt]
    		\begin{axis}[
    			myaxisstyle,  
    			title = {(b) accuracy},
    			xmin=1, xmax=1.075,
    			]
				\addplot[violet] table[x=x, y=c1]{\loadedtable};
				\addplot[orange] table[x=x, y=c2]{\loadedtable};
				\addplot[cyan] table[x=x, y=c3]{\loadedtable};
				\addplot[teal] table[x=x, y=c4]{\loadedtable};
				%\addplot[teal] table[x=x, y=c7]{\loadedtable};
				%\addplot[blue] table[x=x, y=c9]{\loadedtable};
				 %\addplot[green] table[x=x, y=c14]{\loadedtable};
    		\end{axis}
    	\end{tikzpicture}
    \end{minipage}
    \hfill
    \begin{minipage}[b]{0.32\linewidth}
	\pgfplotstableread[col sep=comma]{data/MR_DsDg_minres.csv}\loadedtable
	\begin{tikzpicture}[font=\footnotesize,inner sep=2pt, outer sep=0pt]
	\begin{axis}[
		myaxisstyle,  
		title = {(c) MINRES iterations},
		xmin=1, xmax=5,
		]
		\addplot[violet] table[x=x, y=c1]{\loadedtable};
		\addplot[orange] table[x=x, y=c2]{\loadedtable};
		\addplot[cyan] table[x=x, y=c3]{\loadedtable};
		\addplot[teal] table[x=x, y=c4]{\loadedtable};
		%\addplot[teal] table[x=x, y=c7]{\loadedtable};
		%\addplot[blue] table[x=x, y=c9]{\loadedtable};
		 %\addplot[green] table[x=x, y=c14]{\loadedtable};
	\end{axis}
	\end{tikzpicture}
    \end{minipage}
    \begin{minipage}[b]{\linewidth}
   	\centering
   	\begin{tikzpicture}
   	\draw[violet] (0.5,0) -- (1.0,0) node[right] {\Ro-$\Dg$: $[\omega^l, |\tau^z|]$};
   	\draw[orange] (5.0,0) -- (5.5,0) node[right] {\Ro-$\Dg$-ES: $[\omega^l, |\tau^z|]$};
   	\draw[cyan] (10,0) -- (10.5,0) node[right] {\Tu};
   	\draw[teal] (12.5,0) -- (13,0) node[right] {\Tu-ES};
   	\end{tikzpicture}
   \end{minipage}
    \caption{Performance profiles for \Ro-$\Dg$ and \Tu~with and without Algorithm~\ref{alg_es}.
    }
    \label{fig:perf_earlystop}
\end{figure*}

		We reduce the computational time required by the linear solver by stopping earlier. Specifically, instead of allowing a maximum of 50 iterations as before, we switch between 10, 30 and 50 iterations, respectively, depending on whether \Ro~makes good progress or not.
		The details are provided below in Algorithm~\ref{alg_es}.
		
		\begin{algorithm2e}[h!]
			\SetAlgoRefName{ES}
			\DontPrintSemicolon
			\caption{Early stopping criteria for linear solver; here, $\iota_k$ is the number of allowed \minres~iterations at the $(k+1)$-th iteration of \ref{alg_ROSE}}
			\label{alg_es}
			\KwIn{$0 < \eps_1 < \eps_0 << 1$, $0 < \eta_0 < \eta_1 < \eta_2$}
			\lIf{$|\CJ(x_{k+1})-\CJ(x_k)| \leq \eps_1 |\CJ(x_k)|$\label{line_eps1}}{let $\iota_k := \eta_2$ \tcp*[f]{small progress $\to$~many iterations}
			}
			\lElseIf{$|\CJ(x_{k+1})-\CJ(x_k)| \leq \eps_0 |\CJ(x_k)|$}{let $\iota_k := \eta_1$ \label{line_eta1}\tcp*[f]{medium progress $\to$~in between}
			}
			\lElse{let $\iota_k := \eta_0$\tcp*[f]{large progress $\to$~few iterations}
			}
		\end{algorithm2e}
				
		The underlying idea in Algorithm~\ref{alg_es} is that while far away from a local minimum, a crude approximation of the search direction is enough to obtain a sufficient decrease in the objective value. Therefore, in this case \minres~stops after only $\eta_0$ iterations, where $\eta_0=10$ in our experiments. As the rate of change in the objective function value decreases, the maximal number of \minres~iterations is increased to obtain a more accurate estimate of the search direction.
		
		\Cref{fig:perf_earlystop} compares \Ro-$\Dg$ to \Tu, both with and without Algorithm~\ref{alg_es}.
		As displayed, \ref{alg_ROSE} is faster than \Tu~on approximately $80\%$ of the problems with almost similar performance in terms of accuracy. In contrast to \ref{alg_ROSE}, \Tu~does not benefit from using \ref{alg_es}, suggesting that the lower approximation quality of the seed matrix in \Tu~combined with even earlier stopping produces descent directions of poor quality. 
		
		\paragraph{Overall performance} \Cref{tab:perf} reports the total run-time and average target registration error on the 22 image registration problems. The two variants of \Ro-$\Dg$ clearly outperform the two variants of \Tu. In particular, the best method \Ro-$\Dg$-ES improves meaningfully over the best variant of \Tu. We recall from \cite{AMM24} that this variant of \Tu~outperforms standard \LBFGS~by a wide margin on the 22 problems under consideration.
		
		\begin{table}[h!]
			\footnotesize
			\caption{Performance table}
			\label{tab:perf}
			\centering
			\begin{tabular}{ccccc}
				\toprule
				Measures & \Ro-$\Dg$ & \Ro-$\Dg$-ES & \Tu & \Tu-ES \\
				\midrule
				total run-time (sec.) & 624 & \textbf{614} & 904 & 949 \\
				\midrule
				average TRE & 0.537 & \textbf{0.536} & 0.538 & 0.542 \\
				\bottomrule
			\end{tabular}
		\end{table}

		\paragraph{Convergence rate}
		
		%%%%%%%%%%%%%%%%%%%%%%%
%%FIGURES
%%%%%%%%%%%%%%%%%%%%%%%

\newcommand{\cyclelist}{% list stored as a macro
	{violet},%2
	{orange},
	{cyan},
	{teal}
}

\pgfplotsset{
	compat=1.15,
	my axis style/.style={
		xlabel = $k$,
		xscale = 0.5,
		yscale = 0.7,
		axis y line=left, axis x line =bottom,
		cycle list/.expanded={\cyclelist},
		every axis title/.style={above, at={(current axis.north)}, yshift = 0cm},
		every axis y label/.style={at={(current axis.west)},xshift=-1cm, rotate=90},
		every axis x label/.style={at={(current axis.south)},yshift=-0.7cm},
	},
}

\newcommand{\myBoxplot}[4]{
	\begin{tikzpicture}
	\begin{axis}[my axis style,#3,ymode=log,
	%xlabel=,xticklabels={,},xmin=0,xmax=9,xtick={1,...,7,8}
	]% table
	\addplot +[mark=none,solid] table[x index=0, y index=#4] {data/#1_1.dat};
	\addplot +[mark=none,solid] table[x index=0, y index=#4] {data/#1_2.dat};
	\addplot +[mark=none,solid] table[x index=0, y index=#4] {data/#1_3.dat};
	\addplot +[mark=none,solid] table[x index=0, y index=#4] {data/#1_4.dat};
	%\draw [thin, dashed, draw=black] (axis cs: 0,1) -- (axis cs: 8,1);
	\end{axis}
	\end{tikzpicture}
}

\newcommand{\myboxplotLegend}{
	\centering
	\begin{tikzpicture}
	\draw[gray,solid] (1.5,0) -- (2,0) node[right] {$\Dg: [\omega^l, |\tau^z|]$} ;
	\draw[blue] (5,0) -- (5.5,0) node[right] {$\Dg: [\omega^l, |\tau^z|]$ - ES};
	\draw[magenta] (8.5,0) -- (9,0) node[right] {adap};
	\draw[cyan] (-1,-1) -- (-0.5,-1) node[right] {adap-ES};
	\end{tikzpicture}
}

\begin{figure}[ht!]
	
	\centering
	%\pgfplotsset{ymax=1.2,ymin=-1,yticklabels={-1,0,1},ytick={-1,0,1}}
	\def\expname{conv}
	\begin{tabular*}{\textwidth}{c @{\extracolsep{\fill}}  c @{\extracolsep{\fill}} c @{\extracolsep{\fill}} c}
		\myBoxplot{\expname}{3}{title = {(i) $\|x_k - x_{k-1}\|$}}{1} & 
		\myBoxplot{\expname}{3}{title = {(ii) $|\CJ(x_k)-\CJ(x_{k-1})|$}}{2}  & 
		\myBoxplot{\expname}{3}{title = {(iii) $\|\nabla \CJ(x_k)\|$}}{3}
	\end{tabular*}

	%\myboxplotLegend
	\begin{tikzpicture}
		\draw[violet] (0.5,0) -- (1.0,0) node[right] {\Ro-$\Dg$: $[\omega^l, |\tau^z|]$} ;
		\draw[orange] (5.0,0) -- (5.5,0) node[right] {\Ro-$\Dg$-ES: $[\omega^l, |\tau^z|]$};
		\draw[cyan] (10,0) -- (10.5,0) node[right] {\Tu};
		\draw[teal] (12.5,0) -- (13,0) node[right] {\Tu-ES};
	\end{tikzpicture}

	\caption{Convergence behavior of \Ro~and \Tu~on one image registration problem.
	The diagonal choices for $D_k$ used in \Ro~result in higher rates of convergence compared to the scalar multiple of the identity used in \Tu.}
	\label{fig:conv}
\end{figure}
		
		In \Cref{fig:conv} we assess the rate of convergence of \ref{alg_ROSE} in comparison to \Tu. 
		Unsurprisingly, the diagonal choice of $D_k$ used in \Ro~enables much faster convergence than the scalar multiple of the identity used in \Tu. 	
		
		\paragraph{Visualization of the registration}
		
		\newcommand{\igg}[1]{\includegraphics[width=0.19\linewidth]{figures/#1.png}}

\begin{figure}
	\setlength{\tabcolsep}{2pt}
	\renewcommand{\arraystretch}{1}
	\footnotesize
	\begin{tabular}{cccccc}
		 & Template & Reference & Overlayed T \& R & \multicolumn{2}{c}{Overlayed $T(\phi)$ \& R} \\
		 &   (T)      &  (R)      &                                  & \Ro-$\Dg$ & \Tu \\
		\rotatebox{90}{TC-18: 3D-Lung} & \igg{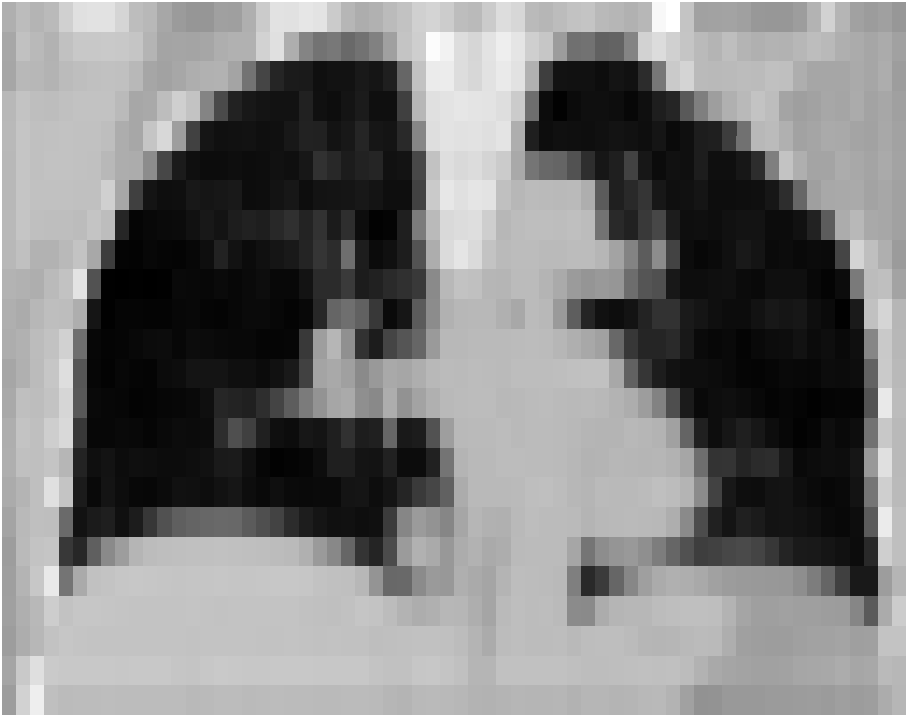} & \igg{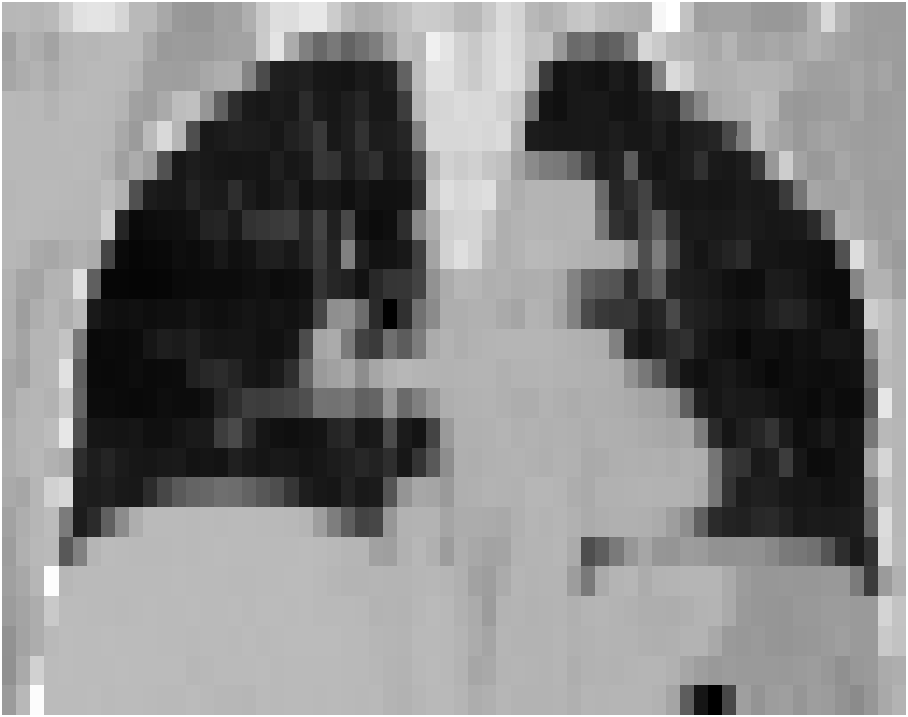} & \igg{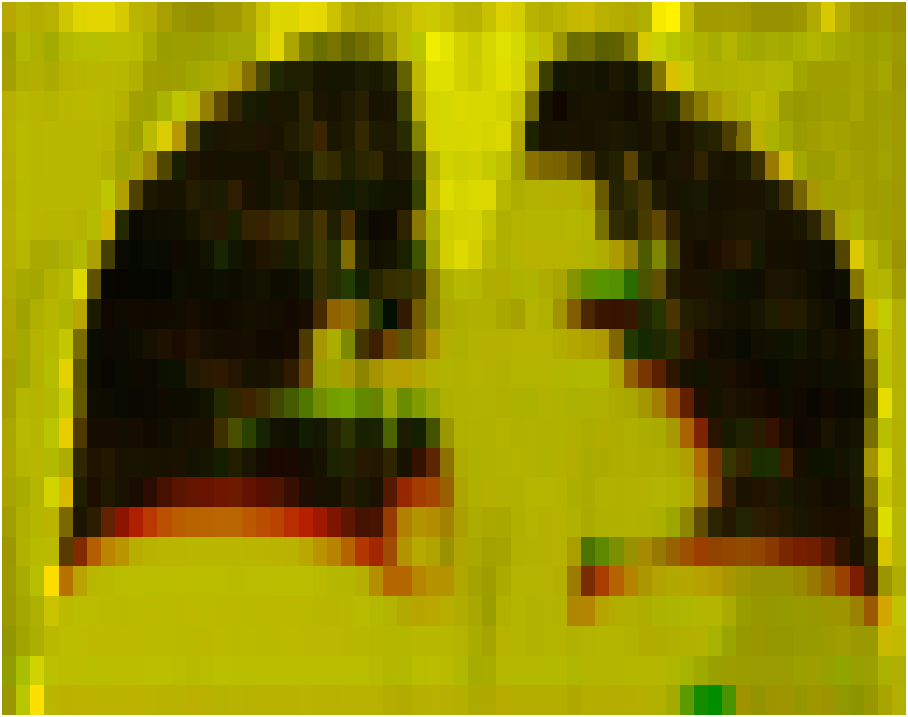} & \igg{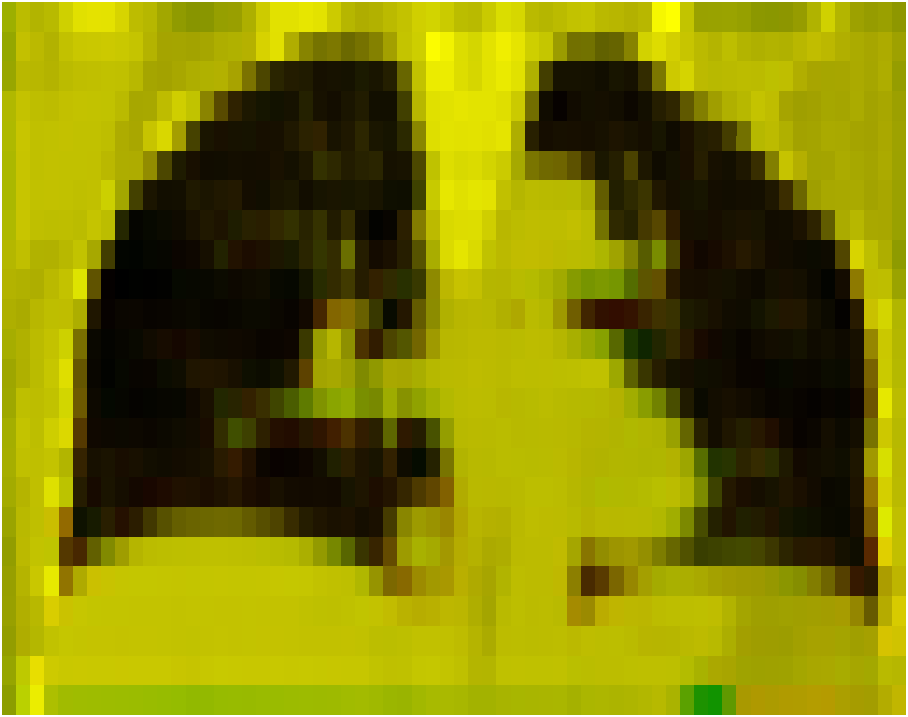} & \igg{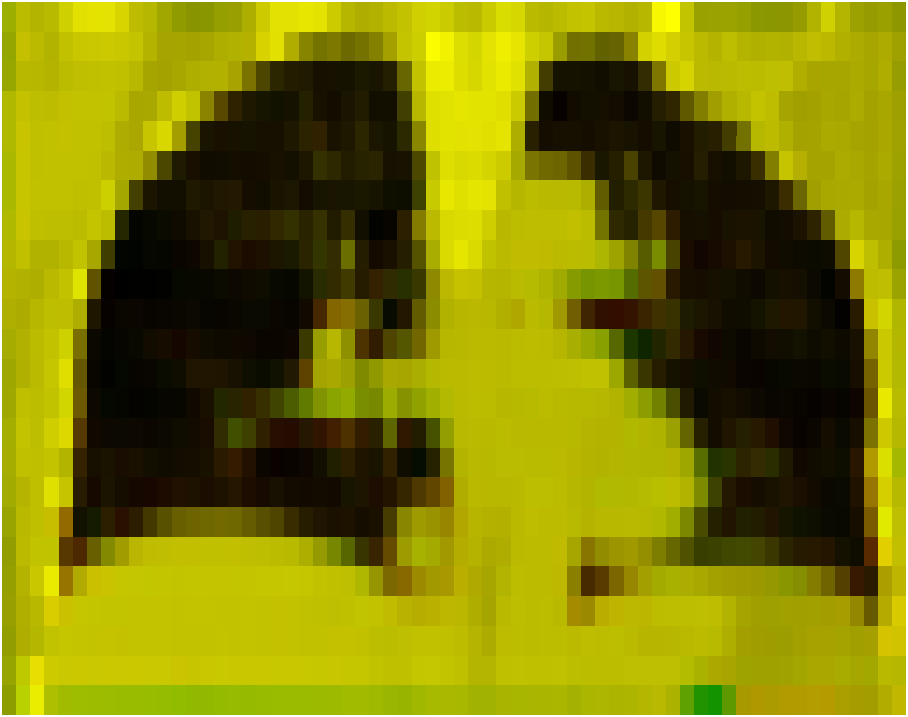} \\
		& & & & TRE: $1.52$, RED: $0.946$ & TRE: $1.53$, RED: $0.947$  \\
		\rotatebox{90}{TC-9: 2D-Hands}& \igg{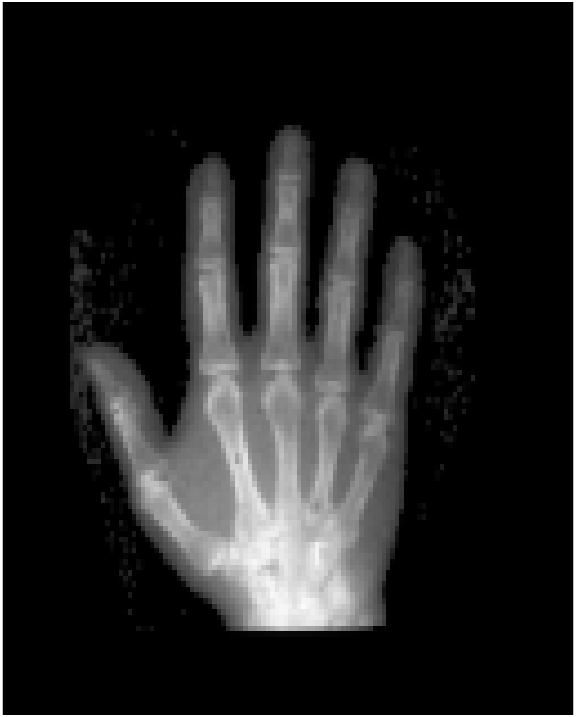} & \igg{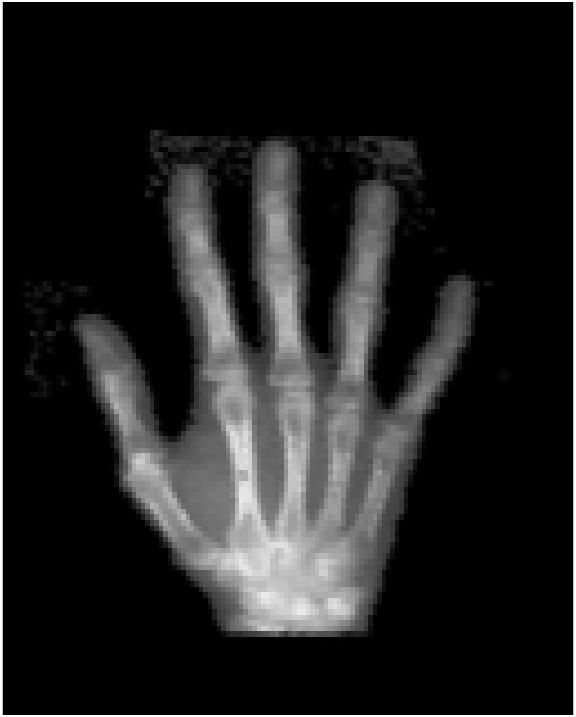} & \igg{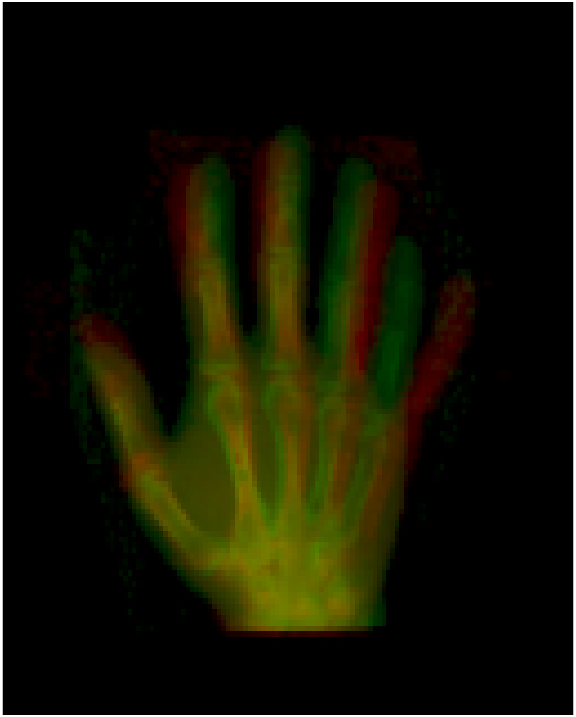} & \igg{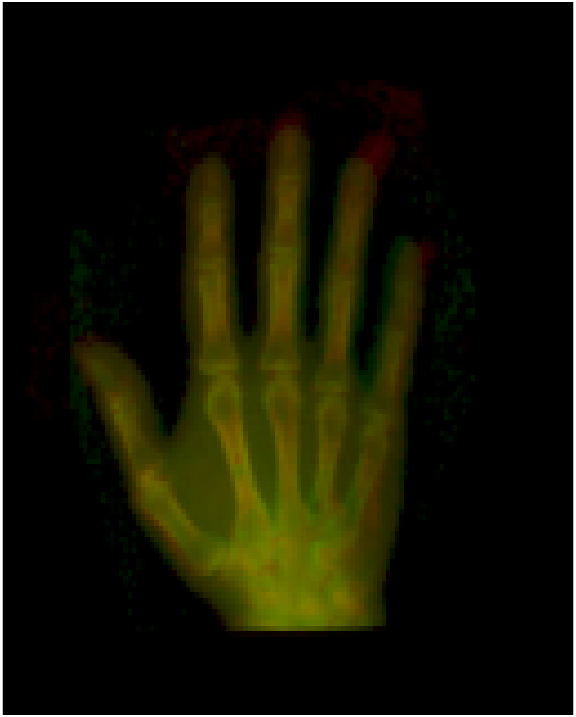} & \igg{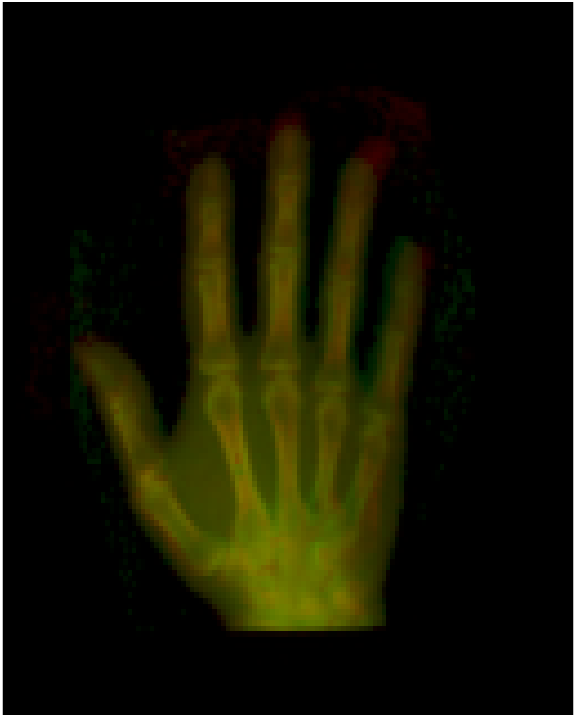} \\
		& & & & TRE: $0.74$, RED: $0.712$ & TRE: $0.72$, RED: $0.715$ \\
		\rotatebox{90}{TC-17: 3D-Brain}& \igg{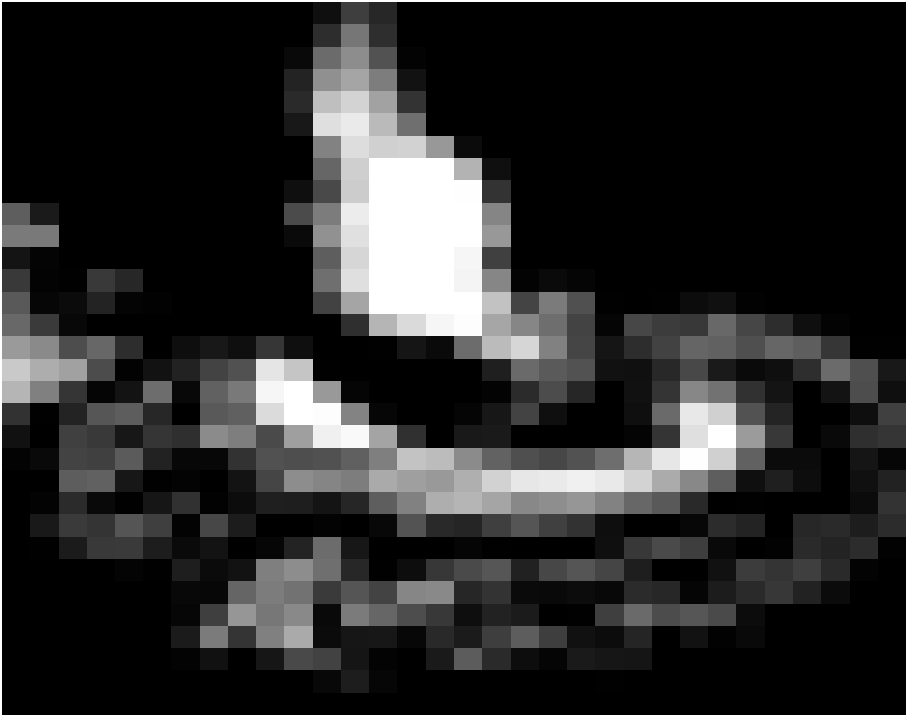} & \igg{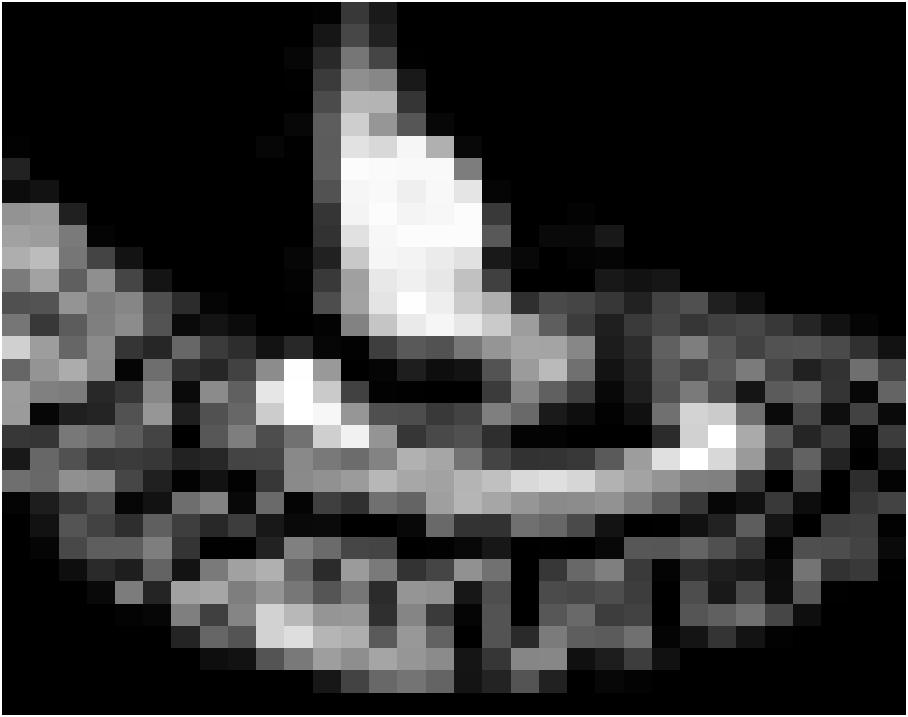} & \igg{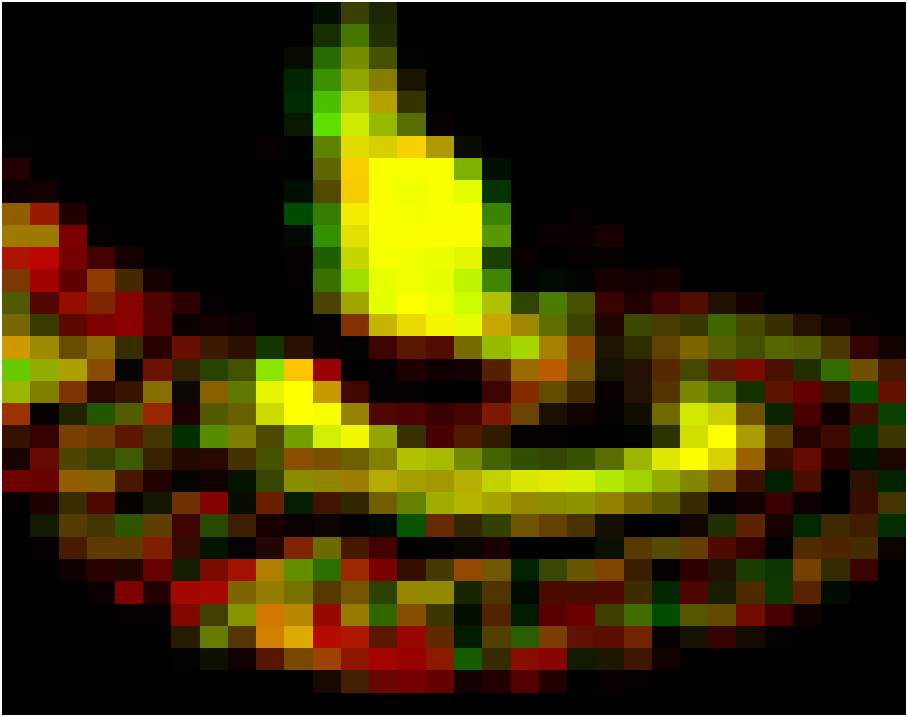} & \igg{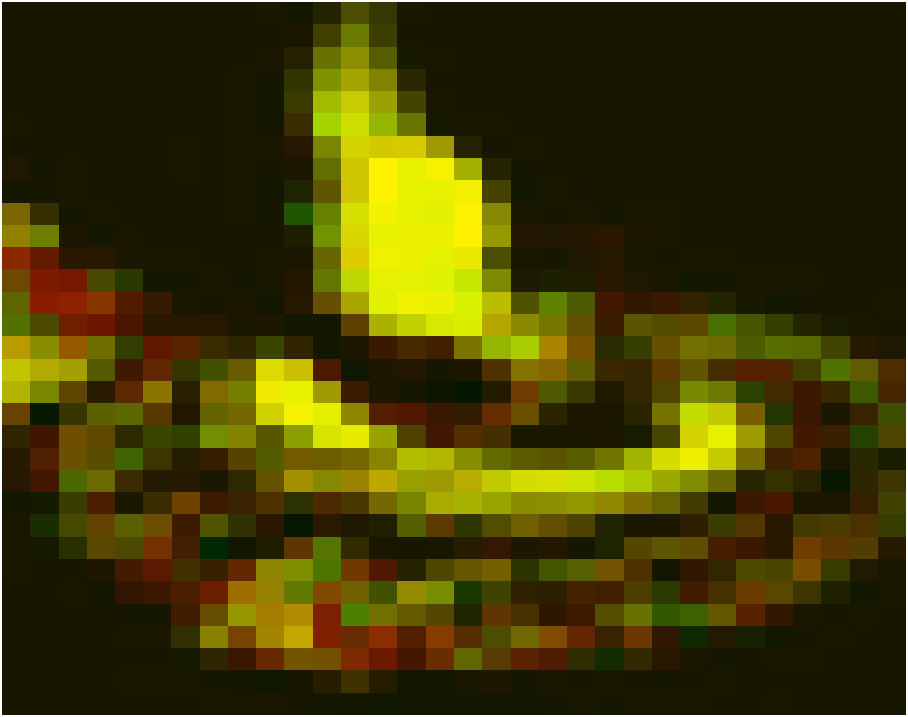} & \igg{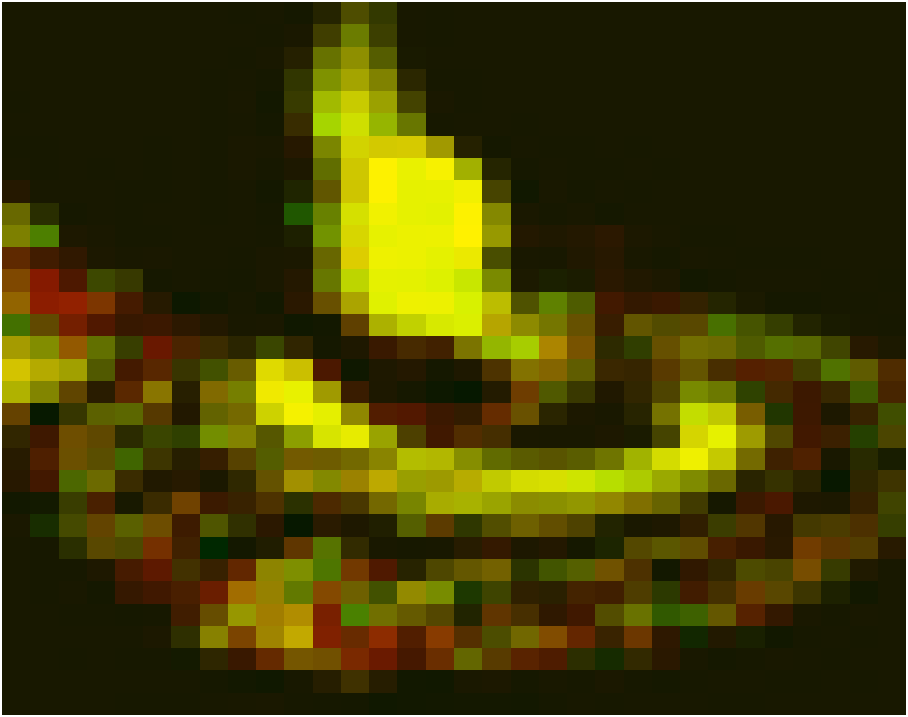} \\
		& & & & RED: $0.4260$ & RED: $0.4247$ \\
		\rotatebox{90}{TC-15: 2D-MRI-Head}& \igg{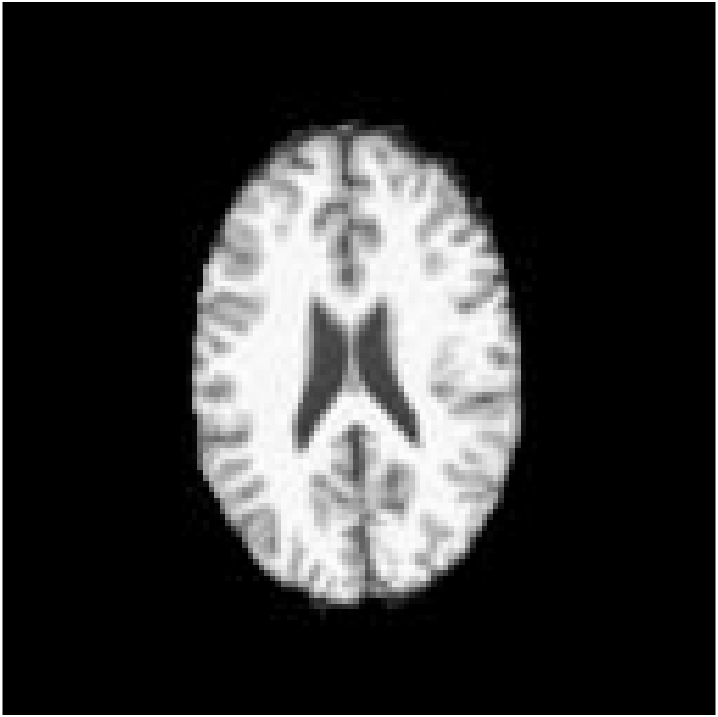} & \igg{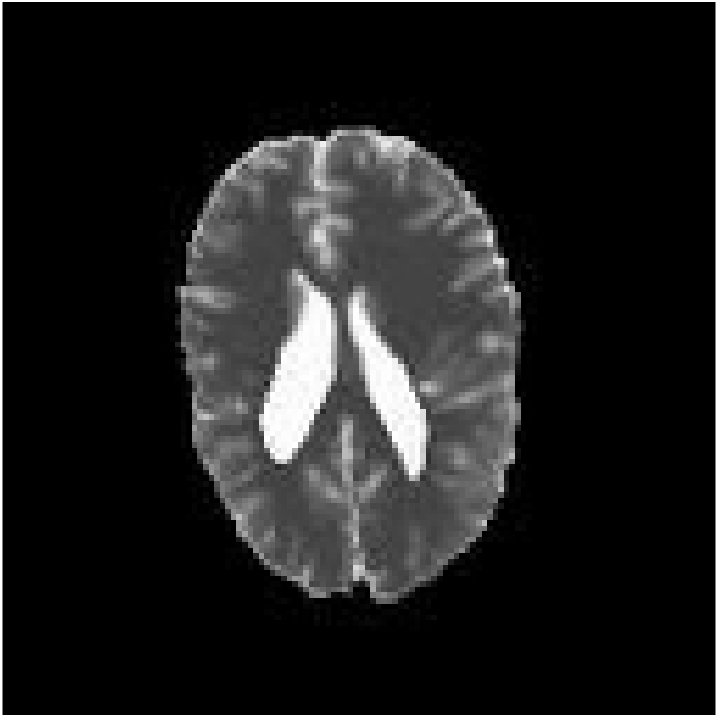} & \igg{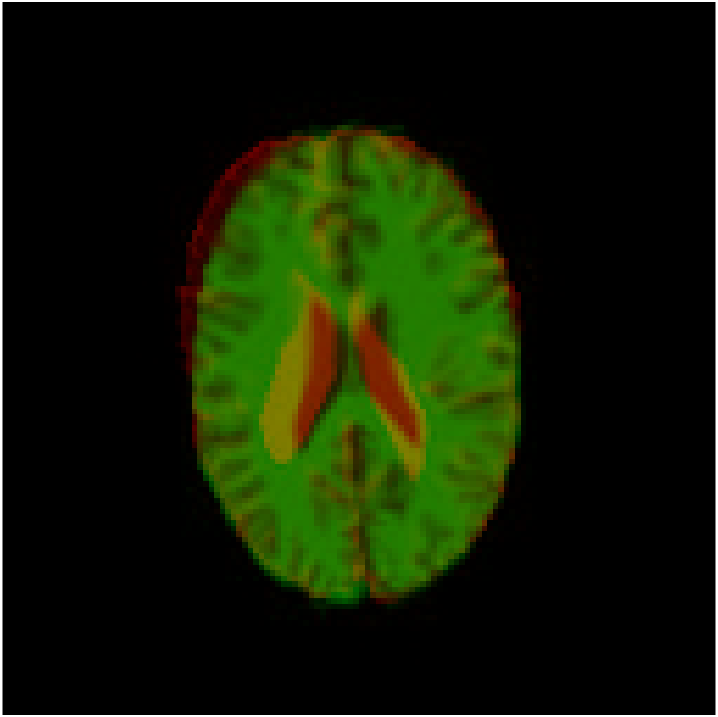} & \igg{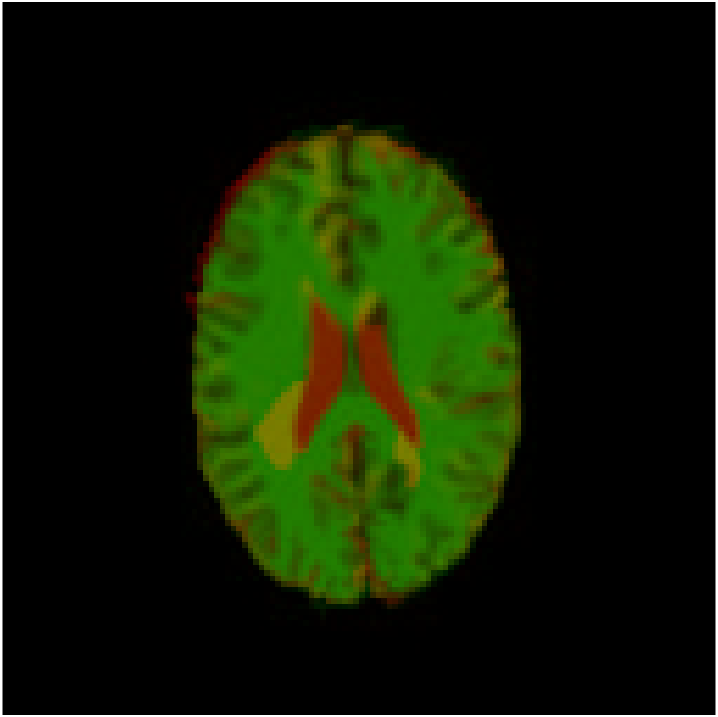} & \igg{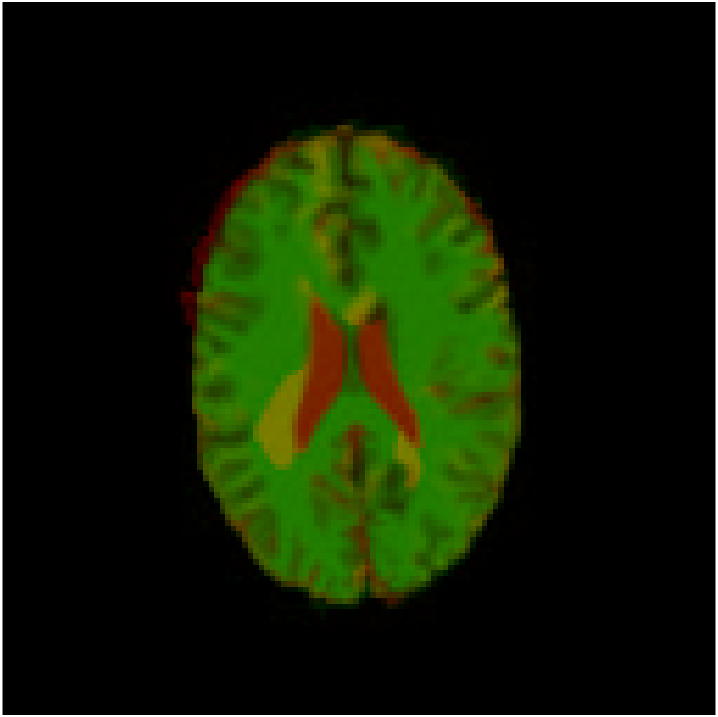} \\
		& & & & RED: $0.9937$ & RED: $0.9938$ \\
		\rotatebox{90}{TC-13: 2D-PET-CT}& \igg{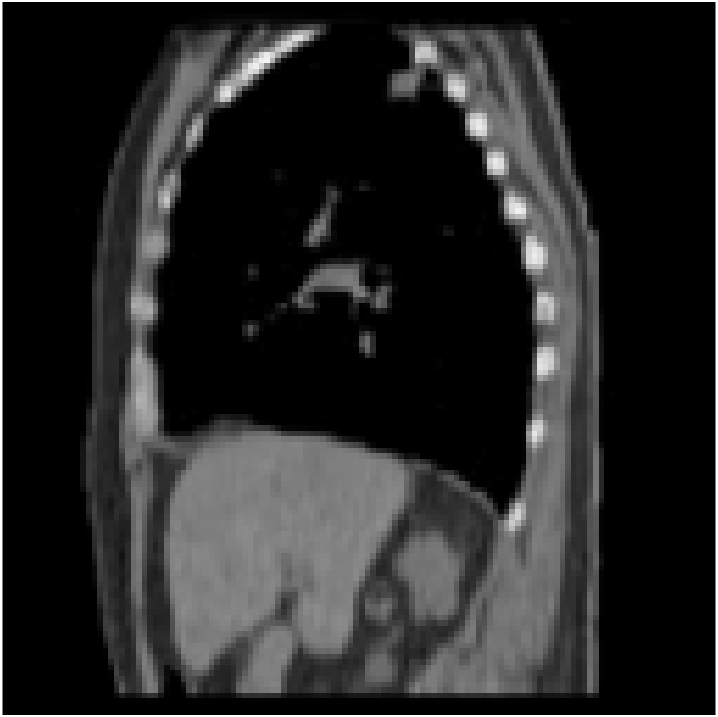} & \igg{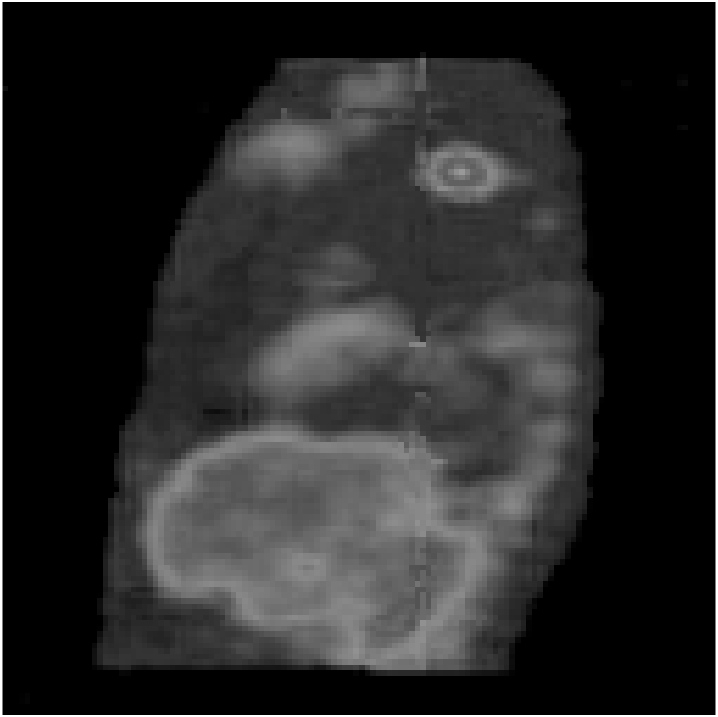} & \igg{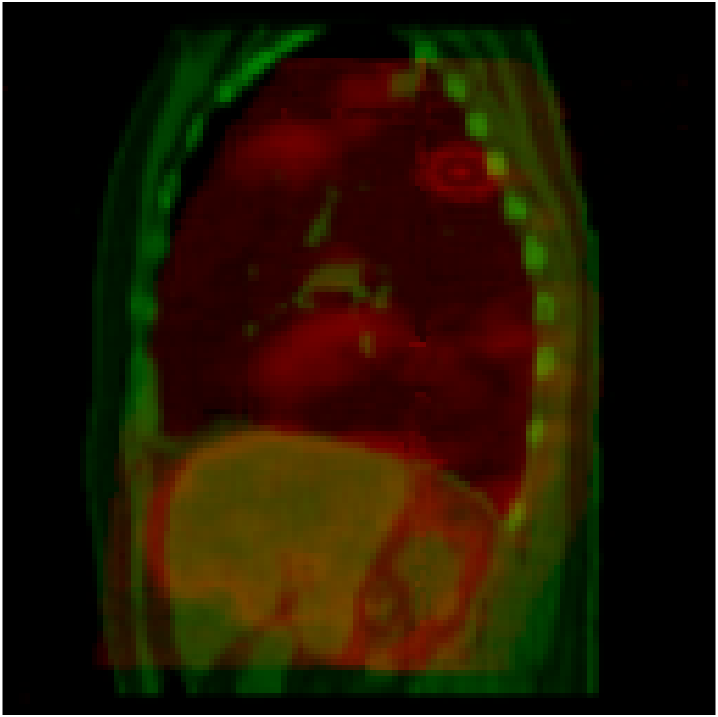} & \igg{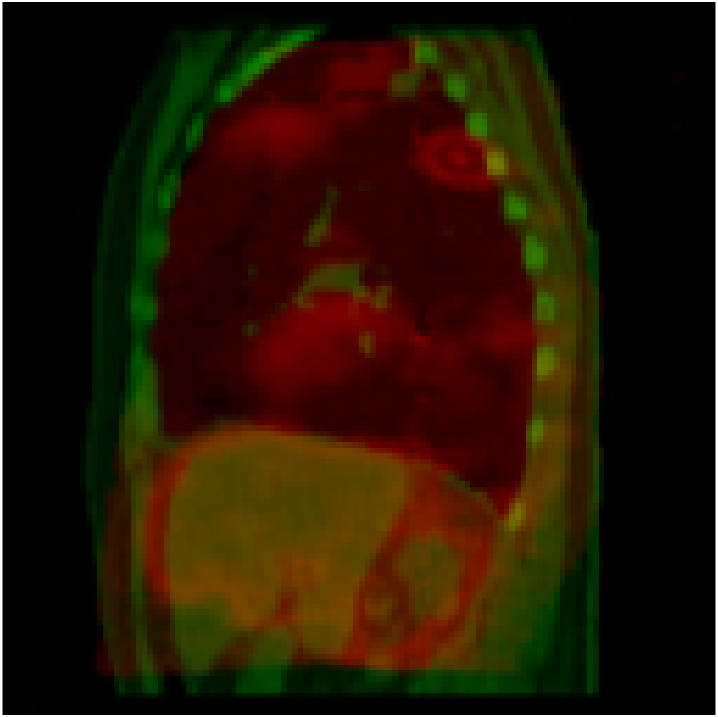} & \igg{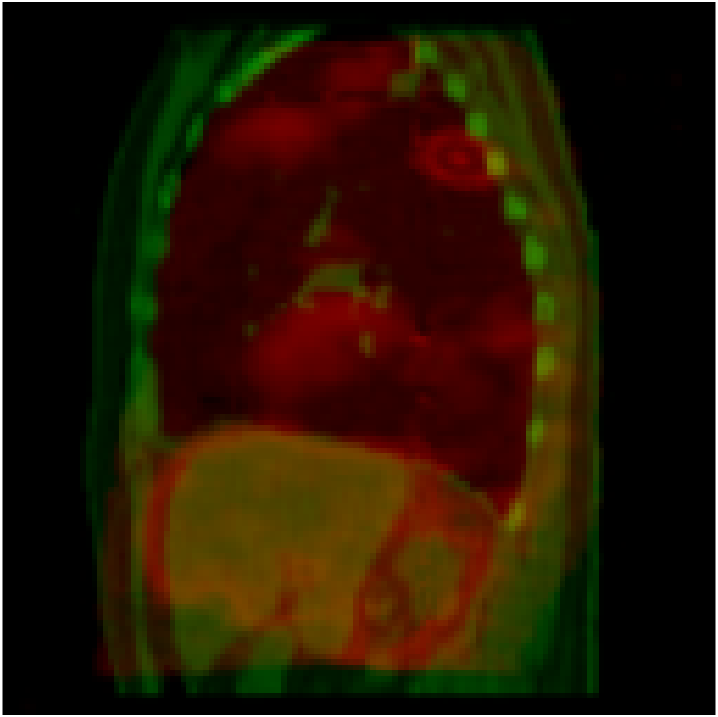} \\
		& & & & RED: $0.9896$ & RED: $0.9896$
	\end{tabular}
	\caption{Registration results for the test cases (TC) 18, 9, 17, 15 and 13; see \Cref{tab:IRproblems}. The first two columns show Template (T) and Reference (R) images, the third column and the last two columns show the difference between T and R before and after registration, respectively. Solution accuracy measure TRE is available only for TC-18 and TC-9, where \Ro-$\Dg$ and \Tu~achieve similar values. 
	The relative objective value reductions (RED) of \Ro-$\Dg$ and \Tu~are consistently close to each other. 
	This indicates that the solution quality of \Ro-$\Dg$ is comparable to that of \Tu~although \Ro-$\Dg$ requires significantly less run-time on average.}
	%, given by $\min\{j,1/j\}$ for $j:=\lvert \CJ(x_f)/\CJ(x_0) \rvert$ with $x_f$ the final iterate, indicates that $\Adap$ does more work than $\Hy$ during the registration for all five test cases.}
	\label{fig_imgreg}
\end{figure}

		\Cref{fig_imgreg} displays the two registration results with the largest TRE for \Ro-$\Dg$ and \Tu. 
		It also depicts registration results on three more datasets where landmarks are not available to measure TRE. 
		The objective value reduction in \cref{fig_imgreg} is very similar for \Ro-$\Dg$ and \Tu, which means that the 
		final objective values are close to each other. 
		This suggests that the run-time improvement of \Ro-$\Dg$ over \Tu~does not affect the quality of the registration. 

	\subsection{Quadratic problems}\label{sec:quad}

		We turn to the academic setting of strictly convex quadratics to illustrate what the ideal problem type is for \ref{alg_ROSE}. 
		This helps to explain the outperformance of \Ro~over \Tu~observed for image registration in \Cref{sec_rlprobs}.
		We study how the regularization parameter and the number of \LBFGS~update vectors affect convergence of \ref{alg_ROSE}, 
		and we validate the results of \Cref{sec_finconv}.
		
		\subsubsection{Problem under consideration}
		
		We seek the unique minimizer $\xopt:=(1,1,\ldots,1)^T\in\R^{16}$ of the strictly convex quadratic 
		$\CJ(x):=0.5(x-\xopt)^T \left[D+\alpha S\right](x-\xopt)$, where
		$D$ and $S$ are SPD and $\alpha>0$. 
		We let $D$ be the diagonal matrix $D_{jj}:=\exp(-j)$ with exponentially decaying eigenvalues. 
		This matrix is ill-conditioned with a condition number around $10^7$, 
		reflecting the Hessian of a typical data-fidelity term in inverse problems.
		For $S$ we use the classical five-point stencil finite difference discretization of the Laplacian with zero boundary conditions on the unit square \cite[Section~1.4.3]{Ba16}. 
		We investigate three setups: \emph{weakly} ($\alpha=10^{-5}$), \emph{mildly} ($\alpha=10^{-3}$) and \emph{strongly} ($\alpha = 10^{-1}$) regularized problems. 
		
		\subsubsection{Algorithmic settings} 
		
		We use $x_0=0$, $S_k=\alpha S$, and we terminate if $\|\nabla \CJ(x_k)\| \leq 10^{-13}$. 
		The linear system in the two-loop recursion is solved with \textsc{Matlab's} backslash.
		We use the same algorithmic parameters for \ref{alg_ROSE}~and \Tu~and the same three variants of $\hat T_{k+1}$ as in \Cref{sec_rlprobs}, cf.~\Cref{tab_paramvalues,tab_diag_bounds}. 
		\Ro-$\Ds$ and \Ro-$\Dg$ agree in this example since the formulas \cref{eq_defgamBp} and \cref{eq_defgamBg} yield identical coefficients, cf.~\cref{lem_goodHessianapprox}~1), so we simply call it \ref{alg_ROSE}.
		
		\subsubsection{Results}
		
		Iteration numbers, average number of line search steps and run-times are summarized in \Cref{tab:quadPrb}. 		
		As expected based on the considerations of \Cref{sec_finconv}, the iteration numbers decrease for increasing $\ell$, except for \ref{alg_ROSE}~when going from $\ell=0$ to $\ell=3$ in case of smaller $\alpha$ and a sufficiently large interval $\hat T_{k+1}$, cf. \cref{rem_fconvwp}~2). 
		\Cref{tab:quadPrb} confirms that, as proved in \Cref{lem_terminationaftertwostepsv3}, the problem structure implies that \ref{alg_ROSE}~with $\ell=0$ terminates after two iterations if $\alpha$ is small enough and $\hat T_{k+1}$ is large enough, cf. also \Cref{rem_fconvwp}~1). 
		
		We repeat that in this section we have deliberately chosen a problem structure that is favorable for \ref{alg_ROSE}~($\CD$ and $\CS$ are strongly convex quadratics, $\CD$ has a diagonal Hessian), since our aim is to illustrate in a clear fashion that the better Hessian approximation that \ref{alg_ROSE}~generates can reduce the number of iterations significantly in comparison to \Tu, enhancing the practical performance. 
		Indeed, \ref{alg_ROSE}~consistently requires fewer iterations and less run-time than \Tu. The variants of \ref{alg_ROSE}~with $[\omega^l,\omega^u]$ and $[\omega^l,\lvert\tau^\Bz\rvert]$ are more effective than with $[\lvert\tau^\Bp\rvert,\lvert\tau^\Bz\rvert]$, which is in line with the results from \Cref{sec_rlprobs}. In consequence, the outperformance in run-time of \ref{alg_ROSE}~over \Tu~displayed in \Cref{tab:quadPrb} is substantial for the two variants with larger intervals. In this respect, we recall that the larger the interval, the better $D_{k+1}$ can potentially approximate the spectrum of $\nabla^2\CD$, which is particularly helpful if $\nabla^2\CD$ is positive definite and ill-conditioned. 
		
		In this example, by obtaining a good approximation of $\nabla^2\CD$, \ref{alg_ROSE}~is able to substantially outperform \Tu~for weak and mild regularization. As an aside we note that like \Tu, standard \LBFGS~struggles on such poorly conditioned problems.
		All in all, \ref{alg_ROSE}~is robust across the full range of regularization values $\alpha$. 
		
		\begin{table}[h!]
\footnotesize
\centering
\caption{Comparison of \ref{alg_ROSE}~and \Tu~on quadratic problems with different regularization strengths $\alpha$ and different numbers of update vectors $\ell$. 
\ref{alg_ROSE}~clearly outperforms \Tu~in run-time.
%with $\alpha=10^{-5},10^{-3},10^{-1}$ and $\ell=3,5,10,\infty$.
%
%Reported are the number of iterations until convergence (top), 
%the average number of line-searches per iteration (middle) and
%the run-time (bottom).
%for seven different Hessian initialization strategies (S).%
}\label{tab:quadPrb}

\setlength{\tabcolsep}{4.5pt}
\newcolumntype{R}{>{\raggedleft\arraybackslash}p{6mm}}
\begin{tabular}{l*{5}{R}c@{\hskip2mm}*{5}{R}c@{\hskip2mm}*{5}{R}}
\hline
\multicolumn{18}{c}{number of iterations}	
\\
\hline
% \cline{1-5}
% \cline{7-10}
% \cline{12-15}
 &\multicolumn{5}{c}{$\alpha=10^{-5}$}
&&\multicolumn{5}{c}{$\alpha=10^{-3}$}
&&\multicolumn{5}{c}{$\alpha=10^{-1}$}
\\
	Method\enspace/\enspace$\ell$ 	&$0$     	&$3$     	&$5$     	&$10$    	&$\infty$  	
	&        	&$0$     	&$3$     	&$5$     	&$10$    	&$\infty$  	
	&        	&$0$     	&$3$     	&$5$     	&$10$    	&$\infty$  
\\
\cline{1-6}
\cline{8-12}
\cline{14-18}
 \Tu  &  5000 & 3088 &   2222 &  1106 &   576  && 5000 &  551 &  409 &  177 &  62  &&  284 & 35 &  35 & 22 & 19 \\
 \Ro: $[\omega^l, \omega^u]$ & 2 & 2 & 2 & 2 & 2 && 2 & 2 & 2 & 2 & 2 && 3 & 3 & 3 & 3 & 3\\
 \Ro: $[\omega^l, |\tau^z|]$ & 5 & 12 & 9 & 8 & 8 && 6 & 12 & 9 & 8 & 8 && 10 & 7 & 6 & 6 & 6\\
 \Ro: $[|\tau^s|, |\tau^z|]$ & 107 & 51 & 46 & 47 & 38 && 85 & 47 & 31 & 32 & 23 && 37 & 25 & 20 & 15 & 14\\
% \cline{1-5}
% \cline{7-10}
% \cline{12-15}
\hline
\multicolumn{18}{c}{average number of line searches per iteration}	
\\
\hline
 \Tu  &  1.01 & 1.14 &   1.15 &  1.16 &   1.15  && 1.00 &  1.11 &  1.11 &  1.22 &  1.39  &&  1.02 & 1.09 &  1.06 & 1.05 & 1.00 \\
\Ro: $[\omega^l, \omega^u]$ & 1.00 & 1.00 & 1.00 & 1.00 & 1.00 && 1.00 & 1.00 & 1.00 & 1.00 & 1.00 && 1.00 & 1.00 & 1.00 & 1.00 & 1.00\\
\Ro: $[\omega^l, |\tau^z|]$ &1.40 & 1.08 & 1.00 & 1.00 & 1.00 && 1.33 & 1.08 & 1.00 & 1.00 & 1.00 && 1.10 & 1.00 & 1.00 & 1.00 & 1.00\\
\Ro: $[|\tau^s|, |\tau^z|]$ & 1.62 & 1.29 & 1.39 & 1.15 & 1.03 && 1.25 & 1.11 & 1.03 & 1.09 & 1.00 && 1.05 & 1.04 & 1.05 & 1.00 & 1.00\\
\hline 
\multicolumn{18}{c}{run-time (in milliseconds)}	
\\
\hline
 \Tu  &  543.2 & 340.5 & 248.1 &  127.9 &  69.9  && 542.5 & 61.9 &  46.4 & 21.7 &  8.7  &&  31.7 & 5.2 &  5.1 & 3.7 & 3.5\\
\Ro: $[\omega^l, \omega^u]$ & 2.9 & 1.6 & 1.5 & 1.5 & 1.5 && 1.6 & 1.5 & 1.5 & 1.4 & 1.4 && 1.9 & 1.6 & 1.6 & 1.6  & 1.6\\
\Ro: $[\omega^l, |\tau^z|]$ & 2.3 & 3.0 & 2.5 & 2.4 & 2.3 && 2.1 & 2.9 & 2.4 & 2.3 & 2.3 && 2.6 & 2.2 & 2.1 & 2.0 & 2.0\\
\Ro: $[|\tau^s|, |\tau^z|]$ & 16.6 & 8.5 & 7.8 & 7.9 & 6.7 && 13.5 & 8.1 & 5.4 & 5.8 & 4.5 && 6.2 & 4.7 & 4.0 & 3.3 & 3.2\\
\hline
\end{tabular}
\end{table}	

%%%%%%%%%%%%%%%%%%%%%%%%%%%%%%%%%%%%%%%%%%%%%%%%%%%%%%%%%%%%%%%%%%%%%%%%%%%%%%%
%%%%%%%%%%%%%%%%%%%%%%%%%%%%%%%%%%%%%%%%%%%%%%%%%%%%%%%%%%%%%%%%%%%%%%%%%%%%%%%

\section{Conclusions}\label{sec:conclusion}
	
	We have presented \ref{alg_ROSE}, a structured \LBFGS~scheme that allows the first part of the structured seed matrix to be diagonal.
	We derived two choices for the diagonal part and we compared them to each other numerically. We found that the choice \cref{eq_defgamBg} related to the geometric mean is substantially more effective. 
		
	\Ref{alg_ROSE}~is well suited for structured large-scale optimization problems,
	including many inverse problems. It comes with strong convergence guarantees, ensuring global and linear convergence in Hilbert space
	even for non-convex objective functions and absent invertibility of the Hessian. 
	These convergence results do not require the first part of the seed matrix to be diagonal, but they hold 
	in particular for the two proposed diagonal choices. 
	The underlying assumptions are especially mild if the objective includes a regularizer 
	for which a computationally cheap and uniformly positive definite Hessian approximation is available. 
	
	In the numerical experiments we have demonstrated 
	on large-scale real-world inverse problems from medical image registration 
	that \ref{alg_ROSE}~outperforms the structured \LBFGS~method TULIP from \cite{AMM24}, 
	which in combination with the findings of \cite{AMM24} implies that 
	it also exceeds other structured \LBFGS~methods and standard \LBFGS~on these problems. 
	In comparison to a scaled identity as in \Tu, the diagonal scaling used in \ref{alg_ROSE}~will usually increase the time required by the iterative linear solver to satisfy the same stopping criterion. On the other hand, diagonal scaling improves the quality of the search directions, 
	so earlier stopping of the linear solver is appropriate. 
	If this is taken into account, diagonal scaling outperforms scalar scaling in our experiments in that it obtains solutions of comparable accuracy in lower run-time. 
	\Ref{alg_ROSE}~can be implemented matrix free and an implementation is available at \href{https://github.com/hariagr/SLBFGS}{https://github.com/hariagr/SLBFGS}.
	
	Future work should assess the numerical performance of Algorithm~\ref{alg_ROSE} on additional classes of inverse problems. 
	It would also be worthwhile to incorporate other diagonal scalings in \Ro~and compare their numerical performance. 
	
%%%%%%%%%%%%%%%%%%%%%%%%%%%%%%%%%%%%%%%%%%%%%%%%%%%%%%%%%%%%%%%%%%%%%%%%%%%%%%%%%
%%%%%%%%%%%%%%%%%%%%%%%%%%%%%%%%%%%%%%%%%%%%%%%%%%%%%%%%%%%%%%%%%%%%%%%%%%%%%%%%%
%
%\section*{Acknowledgments}
%
%%%%%%%%%%%%%%%%%%%%%%%%%%%%%%%%%%%%%%%%%%%%%%%%%%%%%%%%%%%%%%%%%%%%%%%%%%%%%%%%%
%%%%%%%%%%%%%%%%%%%%%%%%%%%%%%%%%%%%%%%%%%%%%%%%%%%%%%%%%%%%%%%%%%%%%%%%%%%%%%%%%

\bibliographystyle{abbrvurl}
\bibliography{lit}

\end{document}